\documentclass[10pt]{article}
\usepackage[margin=1in]{geometry}
\usepackage[utf8]{inputenc}
\usepackage{comment}
\usepackage{amsmath,amssymb,amsthm,amsfonts}

\newcommand{\supp}{\mathrm{supp}}
\renewcommand{\Re}{\mathrm{Re}\,}

\usepackage{bbm}
\usepackage{tikz}
\usepackage{framed}
\usepackage{url}
\usepackage[backend=bibtex]{biblatex}
\nocite{*}
\usepackage{xcolor}
\pagecolor{white}
\usepackage{hyperref}

\newtheorem{theorem}{Theorem}[section]
\newtheorem{lemma}[theorem]{Lemma}
\newtheorem{corollary}[theorem]{Corollary}

\theoremstyle{definition}
\newtheorem{definition}{Definition}[section]
\newtheorem*{remark}{Remark}

\numberwithin{equation}{section}
\numberwithin{figure}{section}
\numberwithin{table}{section}

\title{Rigorous High-Order Hausdorff Dimension Estimation of Limit Sets of Continued Fraction Iterated Function Systems via B-Splines}
\author{Jacob Brown\footnote{Department of Mathematics, University of Connecticut, Storrs, CT 06269, USA (e-mail: \href{mailto:jacob.l.brown@uconn.edu}{jacob.l.brown@uconn.edu)}}}
\begin{document}

\maketitle

\begin{abstract}
We develop a method for the rigorous estimation of Hausdorff dimensions of limit sets produced by continued fraction iterated function systems. Our method is based on the approximation of a Perron-Frobenius operator using the finite element method with B-splines as the choice of basis functions. This choice provides key numerical advantages including higher-order convergence and computational flexibility. We prove an analogue of Falk and Nussbaum's result on ``hidden positivity" for B-spline quasi-interpolants to give rigorous upper and lower bounds for the Hausdorff dimensions of various limit sets. We provide numerical results for degree 4 B-spline interpolants in one and two dimensions.
\end{abstract}

\section{Introduction}

Recently, research interest in the fields of fractal geometry and dynamical systems has included the rigorous estimation of the Hausdorff dimensions of limit sets associated to iterated function systems. One of the earliest and most influential results is due to Hutchinson \cite{hutchinson1981fractals}: Let $X$ be a compact metric space and $\{\phi_e\}_{e=1}^m$ be a family of contractive similarities on $X$, with contraction ratios $r_e$, satisfying the open set condition. The Hausdorff dimension of the limit set is the parameter $s>0$ such that $\sum_{e=1}^mr_e^s=1$. This result is sometimes referred to as Hutchinson's formula. Using this formula, one can compute, for example, that the dimension of the $\frac{1}{3}$-Cantor Set is $\log2/\log3$, the dimension of the von Koch curve is $\log4/\log3$, and numerous other examples.

While powerful, Hutchinson's formula does not apply to systems with infinitely many maps, nor to systems whose maps do not consist of similarities. Many of these more complex systems are conformal iterated function systems (CIFSs). Mauldin and Urba\'{n}ski \cite{mauldin1996dimensions} employed the thermodynamic formalism to study these systems, and their methods set a theoretical framework for what would later allow numerical methods to be developed. In particular, they established that there is a unique parameter value $s>0$ such that the Perron-Frobenius operator (to be further discussed in Section 2)
$$L_sf(x)=\sum_{e\in E}\|D\phi_e(x)\|^sf(\phi_e(x))$$
has leading eigenvalue 1 with a strictly positive eigenfunction $f_s$, and that this parameter value is exactly equal to the Hausdorff dimension of the limit set of the CIFS. 

One particular CIFS that has garnered much interest is the \textit{continued fraction} CIFS. Any irrational number $x\in(0,1)$ can be expressed in continued fraction form,
$$x=[e_1(x),e_2(x),e_3(x),\dots]=\cfrac{1}{e_1(x)+\cfrac{1}{e_2(x)+\cfrac{1}{e_3(x)+\cfrac{1}{\dots}}}},\quad\quad e_i(x)\in\mathbb{N},\ i=1,2,\dots$$
There is a natural correspondence between these fractions and the CIFS consisting of $[0,1]$ and the maps
$$\phi_e(x)=\frac{1}{x+e},\quad e\in \mathbb{N}.$$
In particular, for any subset $E\subset\mathbb{N}$ the limit set associate to the system restricted to the maps $\{\phi_e:e\in E\}$ is precisely the set of irrational numbers whose continued fraction expansion only contains elements in $E$. Rigorous dimension estimates for various choices of $E$ has been the subject of numerous publications in recent years. Notably, Pollicott and Vytnova were able to estimate the dimension of the limit set associate to $E=\{1,2\}$ to an accuracy of 200 digits \cite{pollicott2022hausdorff}.

In light of the result of Mauldin and Urba\'{n}ski, numerical methods for approximating the Perron-Frobenius operator and, in turn, the Hausdorff dimensions of various limit sets of CIFSs—and the continued fraction CIFS in particular—were developed. Recently, Falk and Nussbaum developed and iterated upon an approach to approximate the Perron-Frobenius operator using a collocation method based on continuous piecewise multilinear functions \cite{falk2018new,falk2018c}, and subsequently based on continuous piecewise polynomials of arbitrary degree \cite{falk2021hidden}. Their method was employed by Chousionis et al. \cite{chousionis2020dimension} to study the dimension spectra of subsystems of the continued fraction CIFS. Chousionis et al. \cite{chousionis2024rigorous} further developed a finite element method approach using piecewise linear polynomials to obtain rigorous dimension estimates for higher-dimensional continued fractions, the Apollonian gasket, and other systems.

A common obstacle in these numerical methods is the need for positivity preservation in the approximation scheme, which is necessary to obtain rigorous estimates. A result of Korovkin states that any linear approximation scheme that preserves positivity must be at most second-order accurate \cite{korovkin1957order}. Commonly, the Perron-Frobenius operator is approximated by a matrix, and this matrix needs to have nonnegative entries in order for a theorem regarding the spectral radius of positive matrices to hold (See Lemma \ref{lem:pos_mat} in Section 2). 

In \cite{falk2021hidden}, Falk and Nussbaum circumvented the need for positivity in the traditional sense by showing that the Perron-Frobenius operator maps a certain cone of functions into itself, which allowed them to utilize an extension of Lemma \ref{lem:pos_mat}. Through this extension, they were able to utilize polynomial interpolants of degree greater than 1, which do not necessarily preserve positivity and thus lead to negative entries in the approximation matrices. They coined the term ``hidden positivity" to refer to this more abstract notion of positivity.

In this paper, we consider a finite element method approach to estimate the Hausdorff dimensions of continued fraction CIFSs based on B-splines. As we will show, our method has several key advantages. We show that quasi-interpolation of $L_s$ using B-splines of degree $n$ results in convergence of order $n+1$, which allows us to obtain rigorous estimates without turning to high-degree interpolants. Finally, we show that our hidden positivity result extends to B-spline interpolants in all dimensions, allowing us to rigorously estimate dimensions of higher-dimensional systems.

The rest of the paper is structured as follows. In Section 2 we provide background information for B-splines, conformal iterated function systems, and Falk and Nussbaum's hidden positivity. In Section 3 we prove our main theoretical results on hidden positivity for B-splines in both one dimension and higher dimensions. Finally, in Section 4 we provide numerical examples of rigorous upper and lower dimension estimates for various alphabets in one and two dimensions. %We describe the theory underlying the approximation as well as the advantages to using B-splines over other approximation methods. Moreover, we show that Falk and Nussbaum's notion of hidden positivity holds for B-spline approximations of arbitrary degree and dimension. This is necessary because, as described in Section 2, the approximation obtained using B-splines of degree greater than 1 are not necessarily positivity-preserving. Finally, we present numerical results for various continued fraction systems in one and two dimensions.

\section{Preliminaries}

In this section we introduce the necessary background about B-splines and the continued fraction conformal iterated function system. 

\subsection{B-Splines}

We adopt the algorithmic point of view for B-splines presented by de Boor and H\"{o}llig and discussed in \cite{hollig2013approximation}, as opposed to the more traditional definition via divided differences.

\begin{definition}
A \textit{knot sequence} $\xi=\{\xi_k\}_{k\in I}$ is a finite or bi-infinite non-decreasing sequence of real numbers (called \textit{knots}) without accumulation points, where $I$ denotes the index multi-set of the knot sequence. The number of times a knot $\xi_k$ appears in the sequence is called its \textit{multiplicity}, which is denoted by $\#\xi_k$, and if $\#\xi_k\geq2$, we refer to $\xi_k$ as a \textit{repeated knot}. The knot sequence partitions a subset $R\subseteq\mathbb{R}$ into \textit{knot intervals} $[\xi_\ell,\xi_{\ell+1})$. A knot sequence is \textit{uniform} if all of its knot intervals have equal length.
\end{definition}

Figure \ref{fig:knot_sequence} shows an example of a knot sequence. The use of the index $\ell$ for the knot intervals comes from the fact that intervals of the form $[\xi_k,\xi_{k+1})$ can be degenerate in the case when the knot sequence has repeated knots. If the knot sequence has no repeated knots, then the indices $\ell$ and $k$ coincide, in which case we will write $[\xi_k,\xi_{k+1})$ to denote the knot intervals.

\begin{figure}[b]
\centering
\begin{tikzpicture}
            \fill[gray!50] (7,-0.1) rectangle (9,0.1);
            \draw[<->, thick] (0,0)--(10,0);
            \draw (1,0.15)--(1,-0.15);
            \draw (2,0.15)--(2,-0.15);
            \draw (3.5,0.15)--(3.5,-0.15);
            \draw (5,0.15)--(5,-0.15);
            \draw (5.1,0.15)--(5.1,-0.15);
            \draw (4.9,0.15)--(4.9,-0.15);
            \draw (7,0.15)--(7,-0.15);
            \draw (9,0.15)--(9,-0.15);
            
            \node[anchor=south] at (8,0.15) {$[\xi_{\ell},\xi_{\ell+1})$};
            \node[anchor=north] at (0.3,-0.3) {$\dots$};
            \node[anchor=north] at (9.7,-0.3) {$\dots$};
            \node[anchor=north] at (1,-0.15) {$\xi_{-1}$};
            \node[anchor=north] at (2,-0.15) {$\xi_0$};
            \node[anchor=north] at (3.5,-0.15) {$\xi_1$};
            \node[anchor=north] at (7,-0.15) {$\xi_5$};
            \node[anchor=north] at (9,-0.15) {$\xi_6$};
            \node[anchor=north] at (5,-0.15) {$\xi_2=\xi_3=\xi_4$};
            \node[anchor=west] at (10,0) {$\mathbb{R}$};
        \end{tikzpicture}
        \caption{An example of a knot sequence.}
\label{fig:knot_sequence}

\end{figure}
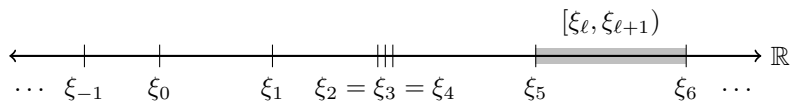

\begin{definition}
Given a knot sequence $\xi=\{\xi_k\}_{k\in I}$, we define the degree $n$ \textit{B-splines} using the following recurrence relation. The degree 0 B-splines are defined as the characteristic functions of the knot intervals,
$$b^0_{k,\xi}(x)=\begin{cases}
1&\text{if $\xi_k\leq x<\xi_{k+1}$}\\
0&\text{otherwise.}
\end{cases}$$
The degree $n$ B-splines for $n\geq 1$ are defined inductively as
$$b^n_{k,\xi}=\gamma^n_{k,\xi}b^{n-1}_{k,\xi}+(1-\gamma^n_{k+1,\xi})b^{n-1}_{k+1,\xi},\quad\text{where }\gamma^n_{k,\xi}(x)=\frac{x-\xi_k}{\xi_{k+n}-\xi_k}.$$
For knot sequences with repeated knots, we discard any terms where the denominator of $\gamma^n_{k,\xi}$ vanishes. It can be shown that the derivative of a B-spline follows the recursion
\begin{equation}
(b^n_{k,\xi})'=\alpha_{k,\xi}^nb^{n-1}_{k,\xi}-\alpha_{k+1,\xi}^{n}b^{n-1}_{k+1,\xi},\quad\text{where }\alpha^n_{k,\xi}=\frac{n}{\xi_{k+n}-\xi_{k}}.
\label{eqn:bsplinederiv}
\end{equation}
\end{definition}

The \textit{parameter interval} $D^n_\xi$ of the B-splines $b^n_{k,\xi}$ is the maximal interval on which the collection of $b^n_{k,\xi}$ form a partition of unity. Figure \ref{fig:param_int} shows an example of a collection of quadratic B-splines and their parameter interval. As can be seen, the parameter interval is a proper subinterval of the entire domain on which the B-splines are constructed. In particular, for degree $n$ B-splines on a finite knot sequence without repeated knots, the parameter interval is obtained by removing $n$ leftmost and $n$ rightmost knot intervals. We require that the parameter interval span the entire domain of the function to be approximated in order for the upcoming approximation results to hold. As such, in implementation, we will need to expand our computational domain by $n$ knot intervals on each side. 

\begin{figure}
\centering
\includegraphics[width=0.5\linewidth]{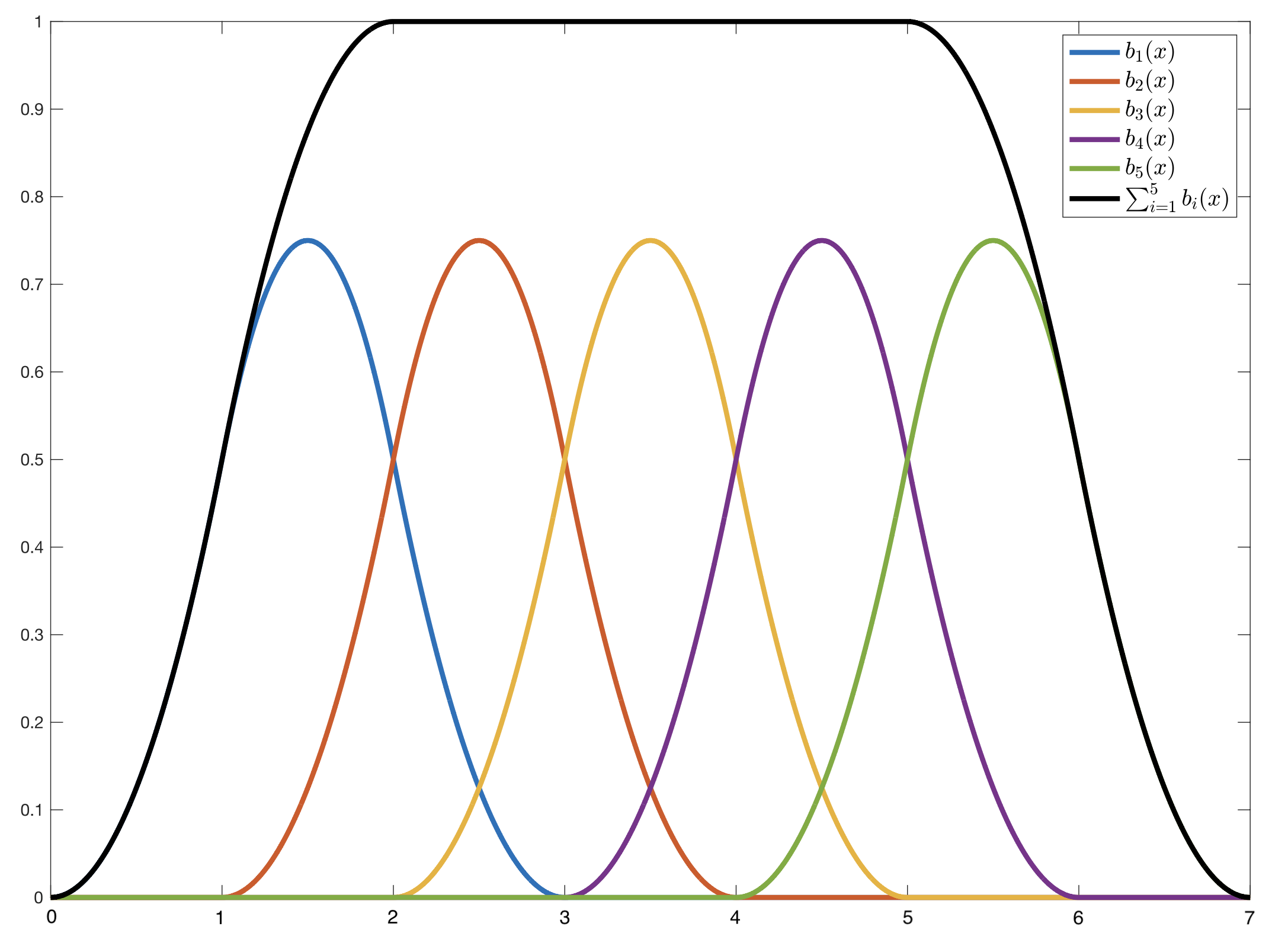}
\caption{The parameter interval $D^2_\xi=[2,5]$ when using quadratic B-splines on the domain $[0,7]$ with knot sequence $\xi=\{0,1,2,\dots,7\}$.}
\label{fig:param_int}
\end{figure}

It follows immediately from the definition that $b^n_{k,\xi}$ is non-negative for all $n$. For our discussion, we will assume that the knot sequence has no repeated knots. Under this assumption, the support of the B-spline is $\supp(b^n_{k,\xi})=[\xi_k,\xi_{k+n+1}]$, that is, the closure of $n+1$ consecutive knot intervals. When the degree and knot sequence are clear, we will write $b_k$ to ease the notation. Given a point $x$, we write $k\sim x$ to denote the B-splines ``relevant" to $x$, i.e., $k\sim x=\{k:b_k(x)>0\}$, and for a set $\Omega$ we write $k\sim\Omega$ to denote the B-splines such that $k\sim x$ for some $x\in\Omega$.

%The parameter interval of degree $n$ B-splines defined on a finite knot sequence $\xi=\{\xi_0,\xi_1,\dots,\xi_m\}$ without repeated knots and where $m\geq 2n$ is $[\xi_n,\xi_{m-n}]$. This fact will be of consequence for the construction of our approximations.

The construction of a $d$-dimensional B-spline is rather straightforward, as it is simply the tensor product of one-dimensional B-splines. To be precise, we first equip each axis $j=1,2,\dots,d$ with a knot sequence $\xi_j=\{\xi_{k_j}\}_{k_j\in I_j}$, where $I_j$ is the index set corresponding to axis $j$. We then construct the degree $n_j$ B-splines $b_{k_j,\xi_j}^{n_j}$ along each axis, and define the $d$-dimensional B-spline as
$$b_{k,\xi}^n(x_1,x_2,\dots,x_d)=\prod_{j=1}^db_{k_j,\xi_j}^{n_j}(x_j),$$
where $n=(n_1,n_2,\dots,n_d)$ is the so-called coordinate degree of the $d$-dimensional B-spline.

As in the univariate case, there is a subset $D^n_\xi$ of the domain on which the collection of B-splines forms a partition of unity, called the \textit{parameter (hyper)rectangle}. Figure \ref{fig:param_int_2d} shows an example of the parameter rectangle for quadratic B-splines in $\mathbb{R}^2$.

\begin{figure}[b]
\centering
\includegraphics[width=0.5\linewidth]{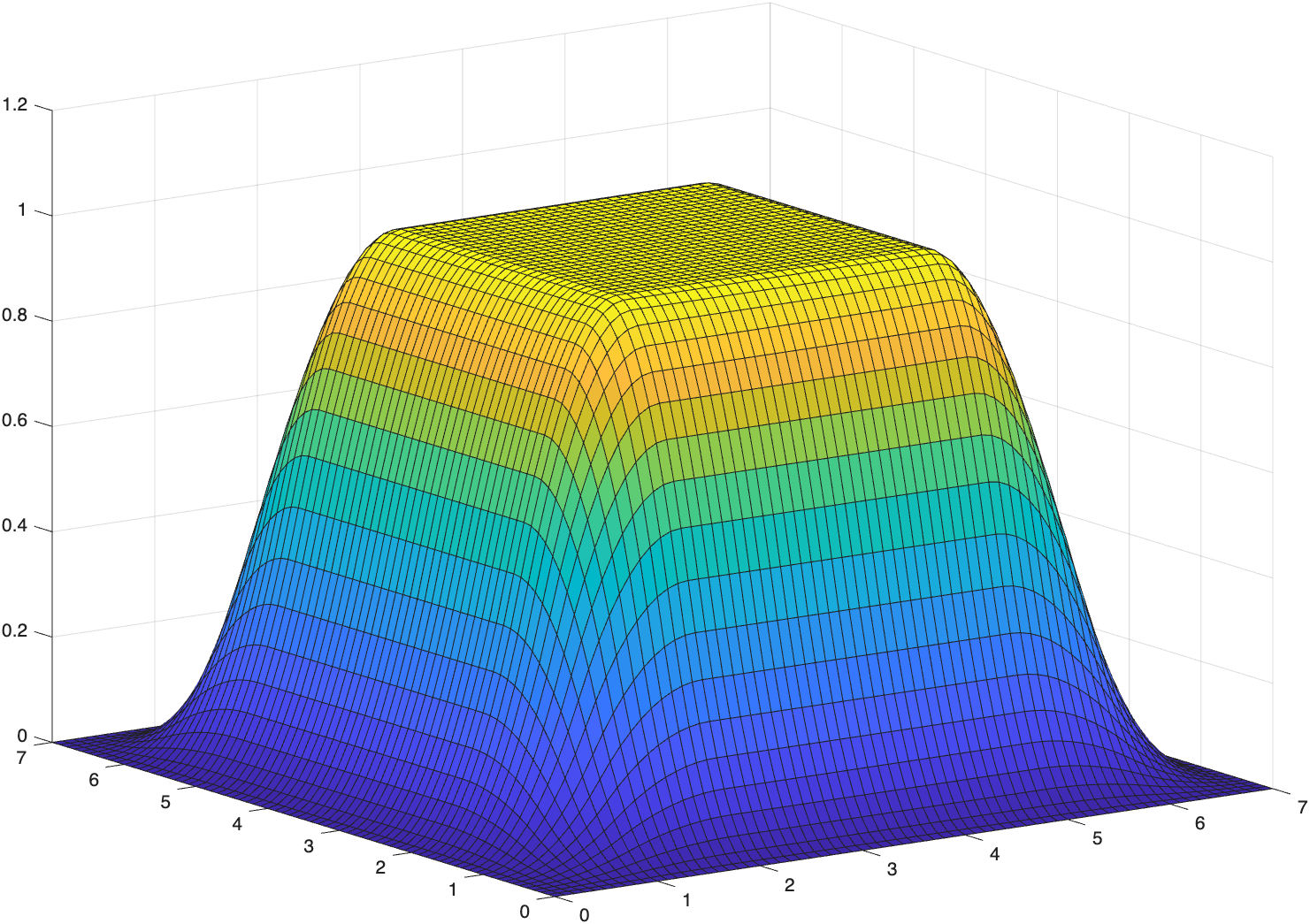}
\caption{Example of the parameter rectangle for 2-dimensional B-splines.}
\label{fig:param_int_2d}
\end{figure}

\clearpage
We now turn our attention to approximation using B-splines.
\begin{definition}
Let $\{b_k\}$ be the collection of degree $n$ B-splines defined on a knot sequence $\xi$ with parameter interval $D^n_\xi$. A linear B-spline approximation scheme of the form $f\approx Qf=\sum_k (Q_kf)b_k$ is called a \textit{quasi-interpolant} provided
\begin{enumerate}
\item Each $Q_k$ is a linear functional that acts on $f$.
\item Each $Q_k$ is locally bounded. That is, there is a number $\|Q\|>0$ such that 
$$|Q_kf|\leq \|Q\|\,\|f\|_{\infty,[\xi_k,\xi_{k+n+1}]}$$
for each $Q_k$, where $\|f\|_{\infty,U}=\sup\{|f(x)|:x\in U\}$.
\item $Qp(x)=p(x)$ for every $x\in D^n_\xi$ and polynomial $p$ of degree at most $n$.
\end{enumerate}
\end{definition}

Quasi-interpolants are desirable for their accuracy in approximating functions and their derivatives, as shown in the following result whose proof can be found on page 86 of \cite{hollig2013approximation}.

\begin{theorem}
\label{thm:quasi_error}
Let $f$ be $n+1$ times continuously differentiable and $Qf=\sum_k(Q_kf)b_k$ be a quasi-interpolant of $f$ using degree $n$ B-splines on a knot interval $\xi$ with parameter interval $D^n_\xi$. For each $x\in D^n_\xi$, let $D_x=\bigcup_{k\sim x}\supp(b_k)$ and $h(x)=\max\{|x-y|:y\in D_x\}$. Then
$$|f(x)-Qf(x)|\leq \frac{\|Q\|}{(n+1)!}\|f^{(n+1)}\|_{\infty,D_x}h(x)^{n+1}.$$
\end{theorem}

We will be using uniform meshes in all of our computational examples. Therefore, if the mesh size is $h$, since $x$ lies in the middlemost interval of $2n+1$ consecutive knot intervals, we have $h(x)\leq(n+1)h$. Replacing $\|f^{(n+1)}\|_{\infty,D_x}\leq\|f^{(n+1)}\|_\infty$, we have the estimate
$$|f(x)-Qf(x)|\leq \frac{(n+1)^{n+1}\|Q\|}{(n+1)!}\|f^{(n+1)}\|_{\infty}h^{n+1}=\frac{(n+1)^n\|Q\|}{n!}\|f^{(n+1)}\|_{\infty}h^{n+1}.$$

As we will see in the coming sections, there are conditions on how small $h$ must be chosen in order for our numerical results to be rigorous. Therefore, we require explicit formulas and/or estimates for every constant in the inequalities above. In particular, we need to be able to compute $\|Q\|$, the boundedness constant of the quasi-interpolant. Below, we will describe our construction of the quasi-interpolant and list the value of $\|Q\|$ for various degrees of the quasi-interpolant. Estimates for $\|f^{(n+1)}\|_{\infty}$ and the other derivatives will be presented in the next subsection.%This number depends on the sample points used 

On a uniform knot sequence with grid size $h$, a natural choice for the quasi-interpolant linear functionals is a weighted sum of the function values at the midpoints of the knot intervals on which $b_k$ is supported. That is,
$$Q_kf=\sum_{v=0}^nw_vf(\xi_{k+v}+h/2),$$
where the weights $w_v$ satisfy the linear system
$$\sum_{v=0}^nw_v(v-j)^n=\prod_{i=1}^n(i-j-1/2).$$
Some important properties of these weights are: (i) $\sum_{v=0}^nw_v=1$, (ii) $w_{\lfloor\frac{n}{2}\rfloor}=\max_v|w_v|$, and (iii) $w_v=w_{n-v}$ for $v=0,1\dots,n$. We remark that a choice for $\|Q\|$ that satisfies the local boundedness requirement for this choice of quasi-interpolant is $\|Q\|=\sum_{v=0}^n|w_v|$. In Table \ref{table:weights}, we provide the weights and the value of $\|Q\|$ for this construction of the quasi-interpolant for $n=1,2,3,$ and $4$.

\begin{table}[h]
\centering
\begin{tabular}{c|c|c|c|c|c|c}
& $w_0$ & $w_1$ & $w_2$ & $w_3$ & $w_4$ & $\|Q\|$ \\\hline
$n=1$ & $1/2$ & $1/2$ & & & & $1$ \\\hline
$n=2$ & $-1/8$ & $5/4$ & $-1/8$ & & & $3/2$ \\\hline
$n=3$ & $-7/48$ & $31/48$ & $31/48$ & $-7/48$ & & $19/12$   \\\hline
$n=4$ & $47/1152$ & $-107/288$ & $319/192$ & $-107/288$ & $47/1152$  & $179/72$ \\
\end{tabular}
\caption{Weights for the 1D quasi-interpolant using the midpoints of the knot intervals for $n=1,2,3,4$ \cite{hollig2013approximation}.}
\label{table:weights}
\end{table}

The generalization of quasi-interpolants to $\mathbb{R}^d$ proceeds similarly to the generalization of B-splines to higher dimensions. We assume that each axis has been equipped with a uniform knot sequence $\xi_j$ with grid size $h_j$. Let $w_{v_j}$, $v_j=0,1,\dots,n_j$ be the weights corresponding to the one-dimensional quasi-interpolant on axis $j$, then the $d$-dimensional quasi-interpolant is defined as $f\approx Qf=\sum_k(Q_kf)b_k$, where the $b_k$ are the $d$-dimensional B-splines and 
\begin{align*}
Q_kf&=Q_{(k_1,k_2,\dots,k_d)}f\\
&=\sum_{v_1=0}^{n_1}\sum_{v_2=0}^{n_2}\cdots\sum_{v_j=0}^{n_j}w_{v_1}w_{v_2}\cdots w_{v_d}f(\xi_{k_1+v_1}+h_1/2,\xi_{k_2+v_2}+h_2/2,\dots,\xi_{k_d+v_d}+h_d/2).
%\label{eqn:d_dim_quasi}
\end{align*}

That is, the coefficient $Q_kf$ is a weighted sum of the function values at the midpoints of the hyperrectangles on which $b_k$ is supported, and the weight on each midpoint is the product of the weights in the one-dimensional quasi-interpolant. As with the univariate case, there is an approximation result for higher-dimensional quasi-interpolants. In order to state this result, we need the following bound on polynomial approximation. We provide the statement as written in \cite{hollig2013approximation}---since we will make a slight modification to it---but the original result is due to Reif \cite{reif2012polynomial}.
\begin{theorem}
\label{thm:poly_bound}
For a hyperrectangle $R=[a_1,b_1]\times\dots\times[a_d,b_d]$, let $\mathbb{P}^n(R)$ denote the space of polynomials of degree at most $n$ on $R$. For a function $f$ on $R$, the error of the orthogonal projection $P^nf\in\mathbb{P}^n(R)$ defined by
$$\int_Rfq\,dx=\int_R(P^nf)q\,dx\quad\text{for all $q\in\mathbb{P}^n(R)$}$$
is bounded by
$$|f(x)-P^nf(x)|\leq c(n,d)\sum_{v=1}^dh_v^{n+1}\|\partial_v^{n+1}f\|_{\infty,R},$$
where $h_v$ is the width of $R$ in the $v$-th direction and $c(n,d)$ is a constant depending on $n$ and $d$.
\end{theorem}

We will now compute the constant $c(n,d)$ following the proof \cite{hollig2013approximation}, but we will make an adjustment that yields a better constant. First, we note that the estimate is invariant under translation and dilation, and thus we may take $R=[0,1]^d$. Let $p_0,p_1,\dots,p_n$ be an orthonormal basis for $\mathbb{P}^n([0,1])$, then by the boundedness of an orthogonal projection we have $\|P^nf\|_{\infty,[0,1]}\leq c_1(n)\|f\|_{\infty,[0,1]}$ where
\begin{equation}
c_1(n)=\sum_{k=0}^n\left(\int_0^1|p_k|\,dx\right)\|p_k\|_{\infty,[0,1]}.
\label{eq:c_1}
\end{equation}
Then for any $p\in \mathbb{P}^n([0,1])$, we have
$$\|f-P^nf\|_{\infty,[0,1]}\leq\|f-p\|_{\infty,[0,1]}+\|P^n(f-p)\|_{\infty,[0,1]}\leq(1+c_1(n))\|f-p\|_{\infty,[0,1]}.$$
By choosing $p$ to be the Taylor polynomial to $f$ at $x=1/2$, we obtain 
$$\|f-p\|_{\infty,[0,1]}\leq \frac{1}{2^n(n+1)!}\|f^{(n+1)}\|_{\infty,[0,1]}.$$
Thus, $\|f-P^nf\|_{\infty,[0,1]}\leq c_2(n)\|f^{(n+1)}\|$ with
\begin{equation}
c_2(n)=\frac{1+c_1(n)}{2^n(n+1)!}.
\label{eq:c_2}
\end{equation}
We refer the reader to \cite{hollig2013approximation} for the remainder of the proof. The constant $c(n,d)$ is given by
\begin{equation}
c(n,d)=c_2(n)\sum_{v=1}^dc_1(n)^{v-1}=\frac{c_2(n)[1-c_1(n)^{d}]}{1-c_1(n)}
\label{eq:multi_error_bound}
\end{equation}

With this result stated, we can now state the error bound for multidimensional quasi-interpolants (see page 148 of \cite{hollig2013approximation}).
\begin{theorem}
\label{thm:quasi_error_higher_dim}
Let $f$ be a smooth function on the parameter rectangle $R=D^n_\xi$. Suppose that each axis $v=1,\dots,d$ is equipped with a knot sequence of uniform mesh size $h$. For any order $n$, there exists a degree-$n$ multivariate spline $p$ such that
$$|f(x)-p(x)|\leq c(n,d)\|Q\|(2n+1)^{n+1}\sum_{v=1}^dh^{n+1}\|\partial_v^{(n+1)}f\|_{\infty,R},$$
where $c(n,d)$ is the constant defined in Equation (\ref{eq:multi_error_bound}). 
\end{theorem}

Note that we may take $p$ to be the quasi-interpolant defined above. %The constant $c(n,d)$ is derived from the  

%\begin{theorem}
%\label{thm:quasi_deriv_error_higher_dim}
%Let $\Omega$ be a region in $\mathbb{R}^d$ and $f\in W^{n+1}_\infty(\Omega)\cap C^{n+1}(\Omega)$. Suppose each axis has been equipped with a uniform knot sequence of mesh size $h$. Let $Qf$ be a quasi-interpolant of $f$, $e=f-Qf$, and $\alpha$ be a multi-index with $|\alpha|\leq n+1$. Then there exists constants $K_\alpha$ such that 
%$$\|D^\alpha e\|_{L^\infty(\Omega)}\leq K_\alpha h^{n+1-|\alpha|}\max_{|\gamma|=n+1}\omega(D^\gamma f;h;\Omega).$$
%\end{theorem} 

\subsection{Conformal Iterated Function Systems}

%Of contemporary interest is the rigorous estimation of the Hausdorff dimensions of fractals that arise from the limit sets of conformal iterated function 

In this section, we will introduce the background of conformal iterated function systems.

\begin{definition}
An \textit{iterated function system} $\mathcal{C}=\{X,E,\Phi\}$ is a triple consisting of a compact metric space $(X,d)$, a countable index set $E$, and a collection $\Phi=\{\phi_e:X\to X\}_{e\in E}$ of injective, uniformly contracting mappings. That is, there is a constant $0<L<1$ such that
$$|\phi_e(x)-\phi_e(y)|\leq L|x-y|$$
for all $e\in E$ and $x,y\in X$. 

An iterated function system $(X,E,\Phi)$ is called \textbf{conformal} provided $X$ is a compact and connected subset of $\mathbb{R}^d$ and the following four conditions are satisfied:
\begin{enumerate}
\item $X=\overline{\mathrm{int}(X)}$.
\item (Open Set Condition) For all $e_1,e_2\in E$ such that $e_1\neq e_2$, $\phi_{e_1}(\mathrm{int}(X))\cap\phi_{e_2}(\mathrm{int}(X))=\emptyset$.
\item There exists an open and connected set $V$ such that $X\subset V\subset\mathbb{R}^d$ such that for each $e\in E$, $\phi_e$ extends to a conformal $C^{1+\varepsilon}$ diffeomorphism on $V$.
\item (Bounded Distortion Property) There exists a constant $K\geq 1$ such that $|D\phi_\omega(x)|\leq K|D\phi_\omega(y)|$ for each $\omega\in E^*$ and $x,y\in X$, where $|D\phi_\omega(x)|$ denotes the maximum norm of the derivative.
\end{enumerate}
\end{definition}

We will adopt the notation of symbolic dynamics for this discussion. A \textit{word} in $E$ is a finite or countably infinite string of elements in $E$. For $n\in\mathbb{N}$, let $E^n$ denote the set of words of length $n$, let $E^*=\bigcup_{n=1}^\infty E^n$ denote the set of finite words, let $E^{\mathbb{N}}$ denote the set of infinite words, and let $E^0=\emptyset$. For $\omega\in E^*$, we use $|\omega|$ to denote the length of $\omega$. For $n\in\mathbb{N}$ and $\omega=\omega_1\omega_2\dots\in E^{\mathbb{N}}$, we let $\omega|_n=\omega_1\omega_2\dots\omega_n$. For a word $\omega=\omega_1\omega_2,\dots,\omega_{|\omega|}\in E^*$, we define $\phi_\omega:X\to X$ by
$$\phi_{\omega}=\phi_{\omega_1}\circ\phi_{\omega_2}\circ\dots\circ\phi_{\omega_{|\omega|}}.$$
Because the collection $\{\phi_e\}_{e\in E}$ is uniformly contracting, for each $\omega\in E^{\mathbb{N}}$, it follows that the sequence of compact sets $\phi_{\omega|_n}(X)$ satisfies $\phi_{\omega|_{n+1}}(X)\subset\phi_{\omega|_n}(X)$ for each $n\in\mathbb{N}$ and $\mathrm{diam}(\phi_{\omega|_n}(X))\to0$ as $n\to\infty$. Thus, Cantor's intersection theorem implies that $\bigcap_{n=1}^\infty \phi_{\omega|_n}(X)$ is a single point, and it is therefore natural to define the projection $\Pi:E^{\mathbb{N}}\to X$ by
$$\Pi(\omega)=\bigcap_{n=1}^\infty \phi_{\omega|_n}(X).$$
There exists a unique set $J$, called the \textit{limit set} associated to $\mathcal{C}$, such that $J=\Pi(E^{\mathbb{N}})$. 

%Rigorous estimation of the Hausdorff dimensions of the limit sets of various conformal iterated function systems
\begin{comment}
To define the Hausdorff dimension of a set, we first need to define the Hausdorff measure on a metric space $(X,d)$.

\begin{definition}
Let $(X,d)$ be a metric space and let $s\geq0$. For $\delta>0$ and $E\subset X$, we define
$$\mathcal{H}^s_\delta(E)=\inf\left\{\sum_{i\in I}\mathrm{diam}(A_i)^s:\text{ $I$ is countable, $A_i\subset X$, $\mathrm{diam}(A_i)<\delta$ for all $i\in I$, and $E\subset\bigcup_{i\in I} A_i$}\right\}$$
Let $\mathcal{H}^s(E)=\lim_{\delta\to0}\mathcal{H}^s_\delta(E)$, then the set function $\mathcal{H}^s(E):\mathcal{P}(X)\to[0,\infty]$ is the \textit{s-dimensional Hausdorff measure} of the set $E$.
\end{definition}

The definition above indeed defines a measure on $(X,d)$, and we refer readers to (CITATION) for the details. We will now define the Hausdorff measure of a set.

\begin{definition}
Let $(X,d)$ be a metric space and let $E\subset X$. The \textit{Hausdorff dimension} $\mathrm{dim}_{\mathcal{H}}(E)$ of $E$ is
\begin{align*}
\mathrm{dim}_{\mathcal{H}}(E)&=\sup\{s\geq0:\mathcal{H}^s(E)>0\}\\
&=\sup\{s\geq0:\mathcal{H}^s(E)=\infty\}\\
&=\inf\{s\geq0:\mathcal{H}^s(E)=0\}\\
&=\inf\{s\geq0:\mathcal{H}^s(E)<\infty\}
\end{align*}
\end{definition}
\end{comment}

Given a conformal iterated function system, a \textit{Perron-Frobenius operator} (also known as a \textit{transfer operator}) on the system is defined as 
\begin{equation}
L_sf(x)=\sum_{e\in E}\|D\phi_e(x)\|^sf(\phi_e(x)),
\label{eq:perron}
\end{equation}
where $x\in X$ and $s>0$ is a parameter. We use $\|\cdot\|$ to denote the supremum (maximum) norm. 

We will now briefly discuss the thermodynamic formalism employed by Mauldin and Urba{\'n}ski to study this operator \cite{mauldin1996dimensions}. Our discussion here will be limited to the results we will use later, and the preliminaries needed to understand them. For a more thorough treatment, we refer readers to the above-cited paper, and the book by the above-cited authors \cite{mauldin2003graph}.

Given a conformal iterated function system, let $J$ denote its limit set and fix a parameter $s\geq0$. A Borel probability measure $m$ on $J$ is said to be \textit{$s$-conformal} provided the following two conditions are satisfied. First, for every $\omega\in E^*$ and Borel set $A\subset X$,
$$m(\phi_\omega(A))=\int_A|D\phi_{\omega}(x)|^s\,dm(x).$$
Second, for any pair of incomparable words $\omega,\tau\in E^*$ (meaning neither word is an extension of the other),
$$m(\phi_\omega(X)\cap\phi_\tau(X))=0.$$
A conformal iterated function system is said to be \textit{regular} if there exists an $s$-conformal measure. We will now partially state an important result by Mauldin and Urba{\'n}ski, whose full statement and proof can be found in \cite{mauldin2003graph}.
\begin{theorem}
Let $\mathcal{C}$ be a regular conformal iterated function system and $m$ be the corresponding $s$-conformal measure. Then there exists a unique continuous function $\rho:X\to[0,\infty)$ such that 
$$L_s(\rho)=\rho\quad\text{and}\quad \int_X\rho\,dm=1,$$
where $L_s$ is the Perron-Frobenius operator (\ref{eq:perron}). Moreover,
$$K^{-s}\leq\rho\leq K^s,$$
where $K$ constant in the bounded distortion property.
\label{thm:CIFS_props}
\end{theorem}
%Under certain assumptions (which are satisfied by the systems at study here), Mauldin and Urba{\'n}ski showed that there exists a parameter $s^*$ such that $L_{s^*}$ has a strictly positive eigenfunction $f_{s^*}$ such that the eigenvalue associated to $f_{s^*}$ is maximal and equal to 1 \cite{mauldin1996dimensions}. Moreover,

Our attention will be focused on systems of continued fractions in $\mathbb{R}$ and $\mathbb{R}^d$ for $d\geq 2$. In 1D, the system is as follows:

Let $X=[0,1]$ with the Euclidean metric and $E\subset\mathbb{N}$. Define $\phi_e:X\to X$ for $e\in \mathbb{N}$ by
$$\phi_e(x)=\frac{1}{x+e}.$$
For computation, we will need bounds on the constant $K$ from the bounded distortion property. If $1\in E$, then $K=4$, otherwise if $\min E=k\geq2$, then $K\leq\exp(\frac{2}{k^2-1})$ \cite{chousionis2020dimension}. Moreover, we need bounds for $\|f^{(j)}_s\|_\infty$, the maximum norm of the $j$-th derivative of the eigenfunction from the Perron-Frobenius operator (\ref{eq:perron}). We will need to modify the following result of Falk and Nussbaum to obtain such bounds \cite{falk2018c}. 
\begin{theorem}
Let $E\subset\mathbb{N}$ be a finite set of natural numbers and let $\gamma=\min E$, $\Gamma=\max E$, and $K=\gamma^{-1}+\Gamma$. Let $A\geq\gamma^{-1}$ and $f_s$ be the eigenfunction from the Perron-Frobenius operator (\ref{eq:perron}) on the domain $[0,A]$. Then for any integer $j\geq1$,
\begin{equation}
(2s)(2s+1)\cdots(2s+j-1)(K+A)^{-j}\leq (-1)^j\frac{f_s^{(j)}(x)}{f_s(x)}\leq (2s)(2s+1)\cdots(2s+j-1)\gamma^{-j},\quad x\in[0,A]
\label{eq:deriv_bounds_1D}
\end{equation}
\label{lem:deriv_bounds_1D}
\end{theorem}

For our numerical examples, we will need to expand the domain of $f_s$ into the negative axis. Let $\hat{X}=[\sigma,A]$ be the expanded domain. Mimicking the proof of Theorem 6.3 in \cite{falk2018c}, for $\omega=e_1e_2\dots e_n\in E^n$, there are numbers $A_n$ and $B_n$ such that
$$\frac{D^j[D\phi_\omega(x)]^s}{[D\phi_\omega(x)]^s}=\frac{D^j[(x+B_n/A_n)^{-2s}]}{(x+B_n/A_n)^{-2s}}=(-1)^j(2s)(2s+1)\cdots(2s+j-1)(x+B_n/A_n)^{-j}.$$
Lemma 6.2 in \cite{falk2018c} implies that $\gamma\leq B_n/A_n\leq K=\Gamma+\gamma^{-1}$ for all $n\geq1$. Then for $\sigma\leq x\leq A$ and $n\geq1$, we obtain
$$(K+A)^{-j}\leq (x+B_n/A_n)^{-j}\leq (\sigma+B_n/A_n)^{-j}\leq(\sigma+\gamma)^{-j}.$$
We therefore obtain the following bound on the expanded domain.

\begin{corollary}
Let $E\subset\mathbb{N}$ be a finite set of natural numbers and let $\gamma=\min E$, $\Gamma=\max E$, and $K=\gamma^{-1}+\Gamma$. Let $A\geq\gamma^{-1}$ and $f_s$ be the eigenfunction from the Perron-Frobenius operator (\ref{eq:perron}) on the domain $[\sigma,A]$, where $\sigma<0$. Then for any integer $j\geq1$,
\begin{equation}
(2s)(2s+1)\cdots(2s+j-1)(K+A)^{-j}\leq (-1)^j\frac{f_s^{(j)}(x)}{f_s(x)}\leq (2s)(2s+1)\cdots(2s+j-1)(\gamma+\sigma)^{-j},\quad x\in[\sigma,A]
\label{eq:deriv_bounds_1D_expanded}
\end{equation}
\label{lem:deriv_bounds_1D_expanded}
\end{corollary}

In $\mathbb{R}^d$, we let $X=\{y\in\mathbb{R}^d:|y-(1/2,0,0,\dots,0)|\leq 1/2\}$, where $|\cdot |$ denotes the Euclidean metric, and we let $E\subset\mathbb{N}\times \mathbb{Z}^{d-1}$. Let $x=(x_1,x_2,\dots,x_d)\in X$ and define $\phi_e:X\to X$ for $e=(e_1,e_2,\dots,e_d)\in E$ as
$$\phi_e(x)=\frac{x+e}{|x+e|^2}=\left(\frac{x_1+e_1}{|x+e|^2},\frac{x_2+e_2}{|x+e|^2},\dots,\frac{x_d+e_d}{|x+e|^2}\right).$$
Again we need bounds on the constant $K$ from the bounded distortion property. In $\mathbb{R}^2$, we have the bound $K\leq4$ \cite{mauldin1996dimensions}, which will be sufficient for our numerical examples.%Here, we can take $K\leq 4$ for all dimensions (cite upcoming paper?). As in the one-dimensional case, we also need bounds for $\|\partial_x^{(j)}f_s\|_\infty$ and $\|\partial_y^{(j)}f_s\|_\infty$. The following result of Falk and Nussbaum—which we also state in less than full generality for simplicity—provides the bounds we need for our discussion of two-dimensional continued fractions in Section 4 \cite{falk2018new}. 

As in the one-dimensional case, we need bounds on the partial derivatives of $f_s$ in terms of $f_s$. In particular, we require bounds for the single partial derivatives $\|\partial_xf_s\|_\infty$ and $\|\partial_yf_s\|_\infty$, as well as bounds for the third-order and fifth-order partial derivatives. In \cite{falk2018new}, the authors proved bounds up to the fourth order and conjectured the bound
\begin{equation}
\frac{|\partial^{(j)}_if_s(x,y)|}{|f_s(x,y)|}\leq \frac{2s(2s+1)(2s+2)\cdots(2s+j-1)}{\gamma^j}
\label{eq:deriv_bounds_general}
\end{equation}
holds for $j\geq1$ and $i=x,y$, where $\gamma>0$ is taken such that $e_1\geq\gamma$ for all maps $\phi_e$. As part of their proof, they showed that the function $G(u,v;s)=(u^2+v^2)^{-s}$ satisfies
$$\frac{\partial_u^{(j)}G(u,v;s)}{G(u,v;s)}=\frac{P_j(u,v;s)}{(u^2+v^2)^j}\quad\text{and}\quad\frac{\partial_v^{(j)}G(u,v;s)}{G(u,v;s)}=\frac{P_j(v,u;s)}{(u^2+v^2)^j},$$
where $P_j(u,v;s)$ is a polynomial that satisfies the recursion
\begin{equation}
P_{j+1}(u,v;s)=(u^2+v^2)(\partial_uP_j(u,v;s))-2(s+j)uP_j(u,v;s).
\label{eq:poly_recursion}
\end{equation}
We prove the conjectured bounds (\ref{eq:deriv_bounds_general}) hold and are sharp in the for the partial derivatives in the $x$ coordinate. First, we prove a closed form for the polynomials $P_j$. 
\begin{lemma}
The polynomials in the recursion (\ref{eq:poly_recursion}) for $j\geq1$ are given by
\begin{equation}
P_j(u,v;s)=\left(\prod_{i=0}^{j-1}(2s+i)\right)\sum_{k=0}^{\lfloor j/2\rfloor}\frac{(-1)^{j-k}j!}{2^k(j-2k)!k!\prod_{i=0}^{k-1}(2s+2i+1)}u^{j-2k}v^{2k}.
\label{eq:poly_closed}
\end{equation}
\end{lemma}
\begin{proof}
We proceed by induction on $j$. The polynomial $P_1(u,v;s)=-2su$ derived in \cite{falk2018new} satisfies the base case. For the inductive step, we substitute (\ref{eq:poly_closed}) into the right-hand side of (\ref{eq:poly_recursion}). 

\begin{align*}
P_{j+1}(u,v;s)&=(u^2+v^2)\left(\left(\prod_{i=0}^{j-1}(2s+i)\right)\sum_{k=0}^{\lfloor j/2\rfloor}\frac{(-1)^{j-k}j!(j-2k)}{2^k(j-2k)!k!\prod_{i=0}^{k-1}(2s+2i+1)}u^{j-2k-1}v^{2k}\right)\\&\quad\quad-2(s+j)\left(\prod_{i=0}^{j-1}(2s+i)\right)\sum_{k=0}^{\lfloor j/2\rfloor}\frac{(-1)^{j-k}j!}{2^k(j-2k)!k!\prod_{i=0}^{k-1}(2s+2i+1)}u^{j-2k+1}v^{2k}.\\
&=\left(\prod_{i=0}^{j-1}(2s+i)\right)\sum_{k=0}^{\lfloor j/2\rfloor}\frac{(-1)^{j-k}j!(j-2k)}{2^k(j-2k)!k!\prod_{i=0}^{k-1}(2s+2i+1)}u^{j-2k+1}v^{2k}\\
&\quad\quad+\left(\prod_{i=0}^{j-1}(2s+i)\right)\sum_{k=0}^{\lfloor j/2\rfloor}\frac{(-1)^{j-k}j!(j-2k)}{2^k(j-2k)!k!\prod_{i=0}^{k-1}(2s+2i+1)}u^{j-2k-1}v^{2k+2}\\
&\quad\quad-2(s+j)\left(\prod_{i=0}^{j-1}(2s+i)\right)\sum_{k=0}^{\lfloor j/2\rfloor}\frac{(-1)^{j-k}j!}{2^k(j-2k)!k!\prod_{i=0}^{k-1}(2s+2i+1)}u^{j-2k+1}v^{2k}\\
&=\underbrace{\sum_{k=0}^{\lfloor j/2\rfloor} A_{j,k}u^{j-2k+1}v^{2k}}_{(I)}+\underbrace{\sum_{k=0}^{\lfloor j/2\rfloor} B_{j,k}u^{j-2k-1}v^{2k+2}}_{(II)},\\
\end{align*}
where
\begin{align*}
A_{j,k}&=\frac{\left(\prod_{i=0}^{j-1}(2s+i)\right)(-1)^{j-k}j!(j-2k)}{2^k(j-2k)!k!\prod_{i=0}^{k-1}(2s+2i+1)}-2(s+j)\frac{\left(\prod_{i=0}^{j-1}(2s+i)\right)(-1)^{j-k}j!}{2^k(j-2k)!k!\prod_{i=0}^{k-1}(2s+2i+1)}\\
&=\frac{\left(\prod_{i=0}^{j-1}(2s+i)\right)(-1)^{j+1-k}j!}{2^k(j-2k)!k!\prod_{i=0}^{k-1}(2s+2i+1)}(2s+2k+j)
\end{align*}
and 
\begin{align*}
B_{j,k}&=\frac{\left(\prod_{i=0}^{j-1}(2s+i)\right)(-1)^{j-k}j!(j-2k)}{2^k(j-2k)!k!\prod_{i=0}^{k-1}(2s+2i+1)}.
\end{align*}

Note that if $j$ is even, 

Our goal is to simplify $(I)+(II)$ into the polynomial $P_{j+1}(u,v;s)=\sum_{k=0}^{\lfloor(j+1)/2\rfloor}C_{j+1,k}u^{j+1-2k}v^{2k}$ where 

$$C_{j+1,k}=\frac{\left(\prod_{i=0}^{j}(2s+i)\right)(-1)^{j+1-k}(j+1)!}{2^k(j+1-2k)!k!\prod_{i=0}^{k-1}(2s+2i+1)}.$$% match those of (\ref{eq:poly_closed}) for $P_{j+1}$. 

There are three cases to consider. 

The first case, for $k=0$, only $(I)$ contributes to the coefficient on $u^{j+1}$, so we must verify $C_{j+1,0}=A_{j,0}$. Noting the empty product $\prod_{i=0}^{-1}(2s+2i+1)=1$, we obtain $A_{j,0}=\left(\prod_{i=0}^{j-1}(2s+i)\right)(-1)^{j+1}(2s+2j)=\left(\prod_{i=0}^{j}(2s+i)\right)(-1)^{j+1}$, which matches the coefficient of $(\ref{eq:poly_closed})$ for the term with $k=0$.

The second case occurs when $k=\lfloor(j+1)/2\rfloor$, and $j$ is odd. In this case, only $(II)$ contributes to the coefficient on $u^{j+1-2\lfloor(j+1)/2\rfloor}v^{2\lfloor(j+1)/2\rfloor}=v^{j+1}$, so we must verify $C_{j+1,\lfloor(j+1)/2\rfloor}=B_{j,\lfloor j/2\rfloor}$. Since $j$ is odd, $\lfloor j/2\rfloor=(j-1)/2$, so we obtain, by multiplying the numerator and denominator by $(2s+2((j-1)/2)+1)=2s+j$,
$$B_{j,\lfloor j/2\rfloor}=\frac{\left(\prod_{i=0}^{j-1}(2s+i)\right)(-1)^{(j+1)/2}j!}{2^{(j-1)/2}((j-1)/2)!\prod_{i=0}^{(j-3)/2}(2s+2i+1)}=\frac{\left(\prod_{i=0}^{j}(2s+i)\right)(-1)^{(j+1)/2}j!}{2^{(j-1)/2}((j-1)/2)!\prod_{i=0}^{(j-1)/2}(2s+2i+1)}$$
Computing $C_{j+1,\lfloor(j+1)/2\rfloor}$, since $\lfloor(j+1)/2\rfloor=(j+1)/2$, we obtain, by dividing the numerator and denominator by $j+1=2((j+2)/2)$,
$$C_{j+1,\lfloor(j+1)/2\rfloor}=\frac{\left(\prod_{i=0}^{j}(2s+i)\right)(-1)^{(j+1)/2}(j+1)!}{2^{(j+1)/2}((j+1)/2)!\prod_{i=0}^{(j-1)/2}(2s+2i+1)}=\frac{\left(\prod_{i=0}^{j}(2s+i)\right)(-1)^{(j+1)/2}j!}{2^{(j-1)/2}((j-1)/2)!\prod_{i=0}^{(j-1)/2}(2s+2i+1)}.$$
%Comparing $B_{j,\lfloor j/2\rfloor}$ and $C_{j+1,\lfloor(j+1)/2\rfloor}$, %In the denominator, we have
%$$2^{(j+1)/2}((j+1)/2)!=(j+1)(j-1)(j-3)\cdots4\cdot 2$$
Thus $C_{j+1,\lfloor(j+1)/2\rfloor}=B_{j,\lfloor j/2\rfloor}$.

%f $j$ is even, only $(I)$ contributes to the coefficient on $u^{j+1-2\lfloor(j+1)/2\rfloor}v^{2\lfloor(j+1)/2\rfloor}=uv^{j}$, so we must verify $C_{j+1,\lfloor(j+1)/2\rfloor}=A_{j,\lfloor j/2\rfloor}$, that is, $C_{j+1,j/2}=A_{j,j/2}$. We have
%$$A_{j, j/2}=\frac{\left(\prod_{i=0}^{j-1}(2s+i)\right)(-1)^{j+1-j/2}j!}{2^{j/2}(j-2(j/2))!(j/2)!\prod_{i=0}^{j/2-1}(2s+2i+1)}(2s+2(j/2)+j)=\frac{\left(\prod_{i=0}^{j-1}(2s+i)\right)(-1)^{j+1-j/2}j!}{2^{j/2}(j/2)!\prod_{i=0}^{j/2-1}(2s+2i+1)}(2s+2j)$$
%and
%$$C_{j+1,j/2}=\frac{\left(\prod_{i=0}^{j}(2s+i)\right)(-1)^{j+1-j/2}(j+1)!}{2^{j/2}(j+1-2(j/2))!(j/2)!\prod_{i=0}^{j/2-1}(2s+2i+1)}=\frac{\left(\prod_{i=0}^{j}(2s+i)\right)(-1)^{j+1-j/2}(j+1)!}{2^{j/2}(j/2)!\prod_{i=0}^{j/2-1}(2s+2i+1)}.$$
%If $j$ is even, then $\lfloor(j+1)/2\rfloor=\lfloor j/2\rfloor=j/2$, so $2k=j$ and $B_{j,\lfloor j/2\rfloor}=0$. 

The third case, for $k=1,\dots,\lfloor(j+1)/2\rfloor-1$ when $j$ is odd and for $k=1,\dots,\lfloor(j+1)/2\rfloor$ when $j$ is even, both $(I)$ and $(II)$ contribute to the coefficient on $u^{j+1-2k}v^{2k}$, so we must verify $C_{j+1,k}=A_{j,k}+B_{j,k-1}$. We compute,
\begin{align*}
A_{j,k}+B_{j,k-1}&=\frac{\left(\prod_{i=0}^{j-1}(2s+i)\right)(-1)^{j+1-k}j!}{2^k(j-2k)!k!\prod_{i=0}^{k-1}(2s+2i+1)}(2s+2k+j)+\frac{\left(\prod_{i=0}^{j-1}(2s+i)\right)(-1)^{j-k+1}j!(j-2k+2)}{2^{k-1}(j-2k+2)!(k-1)!\prod_{i=0}^{k-2}(2s+2i+1)}\\
&=\left(\prod_{i=0}^{j-1}(2s+i)\right)(-1)^{j+1-k}j!\cdot\frac{(2s+2k+j)(j-2k+1)+2k(2s+2k-1)}{2^k(j-2k+1)!k!\prod_{i=0}^{k-1}(2s+2i+1)}\\
&=\left(\prod_{i=0}^{j-1}(2s+i)\right)(-1)^{j+1-k}j!\cdot\frac{(j+1)(j+2s)}{2^k(j-2k+1)!k!\prod_{i=0}^{k-1}(2s+2i+1)}\\
&=\frac{\left(\prod_{i=0}^{j}(2s+i)\right)(-1)^{j+1-k}(j+1)!}{2^k(j-2k+1)!k!\prod_{i=0}^{k-1}(2s+2i+1)}=C_{j+1,k}.
%&=\frac{\left(\prod_{i=0}^{j-1}(2s+i)\right)(-1)^{j+1-k}j!}{2^{k-1}(j-2k-2)!(k-1)!\prod_{i=0}^{k-2}(2s+2i+1)}\\
%&\quad\quad\quad\times\left[\frac{2s+2k+j}{2(j-2k)(j-2k-1)k(2s+2(k-1)+1)}+j-2k-2\right]
\end{align*}
The induction is complete.
%For $A_{j,0}$
\end{proof}
Our next lemma proves that the bound stated in Remark 5.2 of \cite{falk2018new} is true for all $j\geq1$. For the proof, we first establish some notation and special functions. We denote by $x^{(k)}$ the rising factorial
$$x^{(k)}=x(x+1)(x+2)\cdots(x+k-1).$$
For complex numbers $z,z_1,z_2$ with $\Re z,\Re z_1,\Re z_2>0$, the Gamma function $\Gamma(z)$ and Beta function $B(z_1,z_2)$ are, respectively,
$$\Gamma(z)=\int_0^\infty t^{z-1}e^{-t}\,dt$$
and
$$B(z_1,z_2)=\int_0^1t^{z_1-1}(1-t)^{z_2-1}\,dt.$$
The Gamma and Beta functions are related by
$$B(z_1,z_2)=\frac{\Gamma(z_1)\Gamma(z_2)}{\Gamma(z_1+z_2)},$$
and the Gamma function is related to the falling factorial by
$$x^{(k)}=\frac{\Gamma(x+k)}{\Gamma(x)}.$$
\begin{lemma}
For $u\geq\gamma>0$, $s>0$, $i=u,v$, and $j\geq1$,
\begin{equation}
\frac{|\partial_i^{(j)}G(u,v;s)|}{G(u,v;s)}\leq\frac{(2s)(2s+1)\cdots(2s+j-1)}{\gamma^j}.
\label{eq:bound_for_poly}
\end{equation}
Moreover, this inequality is sharp for the $u$ partial derivative case.
\label{lem:bound_for_poly}
\end{lemma}
\begin{proof}
We first prove the case for the partial derivatives in $u$. We compute using (\ref{eq:poly_closed}),
\begin{align*}
\frac{\partial_u^{(j)}G(u,v;s)}{G(u,v;s)}&=\frac{P_j(u,v;s)}{(u^2+v^2)^j}\\
&=\left(\prod_{i=0}^{j-1}(2s+i)\right)\sum_{k=0}^{\lfloor j/2\rfloor}\frac{(-1)^{j-k}j!}{2^k(j-2k)!k!\prod_{i=0}^{k-1}(2s+2i+1)}\frac{u^{j-2k}v^{2k}}{(u^2+v^2)^j}.
\end{align*}
We can write
$$\prod_{i=0}^{k-1}(2s+2i+1)=2^k(s+\tfrac12)^{(k)}$$
and 
$$\frac{j!}{2^k(j-2k)!k!}=2^k\binom{j}{2k}(\tfrac12)^{(k)}.$$
Thus
\begin{align*}
\frac{\partial_u^{(j)}G(u,v;s)}{G(u,v;s)}&=\left(\prod_{i=0}^{j-1}(2s+i)\right)\sum_{k=0}^{\lfloor j/2\rfloor}\frac{(\tfrac12)^{(k)}}{(s+\tfrac12)^{(k)}}(-1)^{j-k}\binom{j}{2k}\frac{u^{j-2k}v^{2k}}{(u^2+v^2)^j}\\
&=\left(\prod_{i=0}^{j-1}(2s+i)\right)\sum_{k=0}^{\lfloor j/2\rfloor}\frac{\Gamma(\tfrac12+k)\Gamma(s+\tfrac12)}{\Gamma(\tfrac12)\Gamma(s+\tfrac12+k)}(-1)^{j-k}\binom{j}{2k}\frac{u^{j-2k}v^{2k}}{(u^2+v^2)^j}.
\end{align*}
Multiplying by $\frac{\Gamma(s)}{\Gamma(s)}$ and using the relation
$$B(k+\tfrac12;s)=\frac{\Gamma(\tfrac12+k)\Gamma(s)}{\Gamma(\tfrac12+k+s)},$$
we obtain,
\begin{align*}
\frac{\partial_u^{(j)}G(u,v;s)}{G(u,v;s)}&=\left(\prod_{i=0}^{j-1}(2s+i)\right)\sum_{k=0}^{\lfloor j/2\rfloor}\frac{\Gamma(s+\tfrac12)}{\Gamma(\tfrac12)\Gamma(s)}B(k+\tfrac12;s)(-1)^{j-k}\binom{j}{2k}\frac{u^{j-2k}v^{2k}}{(u^2+v^2)^j}\\
&=\left(\prod_{i=0}^{j-1}(2s+i)\right)\frac{\Gamma(s+\tfrac12)}{\Gamma(\tfrac12)\Gamma(s)}\sum_{k=0}^{\lfloor j/2\rfloor}\left(\int_0^1t^{k-1/2}(1-t)^{s-1}\,dt\right)(-1)^{j-k}\binom{j}{2k}\frac{u^{j-2k}v^{2k}}{(u^2+v^2)^j}. \quad\quad\quad (*)
\end{align*}
We can write $t^{k-1/2}=t^kt^{-1/2}$ and group the $t^k$ factor with $v^{2k}$. Then, swapping the integral and the sum and breaking $(-1)^{j-k}=(-1)^j(-1)^k$,
\begin{align*}
\frac{\partial_u^{(j)}G(u,v;s)}{G(u,v;s)}&=\left(\prod_{i=0}^{j-1}(2s+i)\right)\frac{\Gamma(s+\tfrac12)}{\Gamma(\tfrac12)\Gamma(s)}\int_0^1t^{-1/2}(1-t)^{s-1}\cdot(-1)^j\sum_{k=0}^{\lfloor j/2\rfloor}(-1)^{k}\binom{j}{2k}\frac{u^{j-2k}(v^2t)^k}{(u^2+v^2)^j}\,dt.
\end{align*}
Using the binomial theorem, we have
\begin{align*}
\sum_{k=0}^{\lfloor j/2\rfloor}(-1)^{k}\binom{j}{2k}u^{j-2k}(v^2t)^k&=\frac12\left((u+iv\sqrt{t})^j+(u-iv\sqrt{t})^j\right)\\
&=\Re\left\{(u+iv\sqrt{t})^j\right\}\\
&=(u^2+v^2t)^{j/2}\cos\left(j\tan^{-1}\frac{v\sqrt{t}}{u}\right).
\end{align*}
Taking absolute values and noting $0\leq t\leq1$, $|(u^2+v^2t)^{j/2}\cos\left(j\tan^{-1}\frac{v\sqrt{t}}{u}\right)|\leq(u^2+v^2)^{j/2}$. Thus, taking the absolute value of $\partial_u^{(j)}G(u,v;s)/G(u,v;s)$ gives
\begin{align*}
\frac{|\partial_u^{(j)}G(u,v;s)|}{G(u,v;s)}&\leq\left(\prod_{i=0}^{j-1}(2s+i)\right)\frac{\Gamma(s+\tfrac12)}{\Gamma(\tfrac12)\Gamma(s)}\int_0^1t^{-1/2}(1-t)^{s-1}\cdot\frac{(u^2+v^2)^{j/2}}{(u^2+v^2)^j}\,dt\\
&=\left(\prod_{i=0}^{j-1}(2s+i)\right)\frac{\Gamma(s+\tfrac12)}{\Gamma(\tfrac12)\Gamma(s)}\frac{B(\tfrac12;s)}{(u^2+v^2)^{j/2}}\\
&=\frac{\prod_{i=0}^{j-1}(2s+i)}{(u^2+v^2)^{j/2}}.
\end{align*}
The maximum value of the quantity in the final line above occurs where $u=\gamma$ and $v=0$, so we obtain 
$$\frac{\partial_u^{(j)}G(u,v;s)}{G(u,v;s)}\leq\frac{\prod_{i=0}^{j-1}(2s+i)}{\gamma^j},$$
as desired. We'll note here that this inequality is sharp because when $u=\gamma$ and $v=0$, we have $|P_j(\gamma,0;s)|=\left(\prod_{i=0}^{j-1}(2s+i)\right)\gamma^j$ so 
$$\frac{|\partial_u^{(j)}G(\gamma,0;s)|}{G(\gamma,0;s)}=\left|\frac{P_j(\gamma,0;s)}{(\gamma^2+0^2)^j}\right|=\frac{\prod_{i=0}^{j-1}(2s+i)}{\gamma^j}$$
\\ \\
For the partial derivatives in $v$, since $\frac{\partial_v^{(j)}G(u,v;s)}{G(u,v;s)}=P_j(v,u;s)/(u^2+v^2)^j$, we can pick up from the equation labeled with $(*)$ in the work above, but with the roles of $u$ and $v$ swapped. That is,
$$\frac{\partial_v^{(j)}G(u,v;s)}{G(u,v;s)}=\left(\prod_{i=0}^{j-1}(2s+i)\right)\frac{\Gamma(s+\tfrac12)}{\Gamma(\tfrac12)\Gamma(s)}\sum_{k=0}^{\lfloor j/2\rfloor}\left(\int_0^1t^{k-1/2}(1-t)^{s-1}\,dt\right)(-1)^{j-k}\binom{j}{2k}\frac{v^{j-2k}u^{2k}}{(u^2+v^2)^j}.$$
Swapping the sum and the integral and grouping $t^k$ with $u^{2k}$ gives
$$\frac{\partial_v^{(j)}G(u,v;s)}{G(u,v;s)}=\left(\prod_{i=0}^{j-1}(2s+i)\right)\frac{\Gamma(s+\tfrac12)}{\Gamma(\tfrac12)\Gamma(s)}\int_0^1t^{-1/2}(1-t)^{s-1}\cdot(-1)^j\sum_{k=0}^{\lfloor j/2\rfloor}(-1)^{k}\binom{j}{2k}\frac{v^{j-2k}(u^2t)^k}{(u^2+v^2)^j}\,dt.$$
Using the binomial theorem as we did above, we can write
$$\sum_{k=0}^{\lfloor j/2\rfloor}(-1)^{k}\binom{j}{2k}v^{j-2k}(u^2t)^k=\Re\left\{(v+iu\sqrt{t})^j\right\}.$$
If $v\neq0$, then $\Re\left\{(v+iu\sqrt{t})^j\right\}=(v^2+u^2t)^{j/2}\cos\left(j\tan^{-1}\frac{u\sqrt{t}}{v}\right)$, which we can bound using the same argument as in the $u$ partial derivative case. Otherwise, if $v=0$, then 
$$|\Re\left\{(iu\sqrt{t})^j\right\}|\leq (u\sqrt{t})^j=(u^2t)^{j/2},$$
which also leads to the same bound as in the $u$ partial derivative case.
\end{proof}
%With these lemmas established, we can now pass the bound above to the eigenfunction $f_s$, as stated in (\ref{eq:deriv_bounds_general}). Our method of proof is identical to the proof of Theorem 5.8 in \cite{falk2018new}, so we refer readers to their paper for the details.
With these lemmas established, we now wish to obtain derivative bounds for the eigenfunction $f_s$. For our numerical results, we will need to consider bounds that accommodate cases where the $x$ coordinate can be negative. The following result utilizes the approach used to prove Theorem 5.8 in \cite{falk2018new}, with minor modifications.

Adopting the definition in \cite{falk2018new}, we say a set $H$ is \textit{mildly regular} if there exist $\eta>0$ and $M\geq1$ such that whenever $x,y\in H$ and $\|x-y\|<\eta$, there exists a Lipschitz map $\psi:[0,1]\to H$ with $\psi(0)=x$, $\psi(1)=y$, and 
$$\int_0^1\|\psi'(t)\|\,dt\leq M\|x-y\|,$$
where $\|\cdot\|$ is any fixed norm on $\mathbb{R}^d$.
\begin{theorem}
For $s>0$, let $f_s$ be the eigenfunction of the Perron-Frobenius operator (\ref{eq:perron}). Let $\hat{X}=\overline{H}$ be the domain of $f_s$, where $H$ is a mildly regular, bounded, open set. If $\hat{X}\cap\{(x,y):x<0\}\neq\emptyset$, let $\sigma=\min\{x:(x,y)\in \hat{X}\}<0$, otherwise let $\sigma=0$. Suppose there exists $\gamma>0$ such that $e_1\geq\gamma>|\sigma|$ for all $e=(e_1,e_2)\in E$. Then 
\begin{equation*}
\frac{|\partial^{(j)}_if_s(x,y)|}{|f_s(x,y)|}\leq \frac{2s(2s+1)(2s+2)\cdots(2s+j-1)}{(\gamma+\sigma)^j}
\label{eq:deriv_bounds_general_expanded}
\end{equation*}
for all $j\geq1$ and $i=x,y$. When $i=x$, the inequality is sharp.
\end{theorem}
\begin{proof}
By (4.9) in \cite{falk2018new}, for any multi-index $\alpha$ it follows that
\begin{equation}
\frac{\partial^\alpha f_s(x,y)}{f_s(x,y)}=\lim_{n\to\infty}\frac{\partial^\alpha(\sum_{\omega\in E^n}|D\phi_\omega(x,y)|^s)}{\sum_{\omega\in E^n}|D\phi_\omega(x,y)|^s},
\label{eq:fs_deriv}
\end{equation}
where $E^n$ is the set of words with length $n$. By Lemma 5.3 in \cite{falk2018new}, for each $\omega\in E^n$, there exist real numbers $\gamma_\omega\geq \gamma$ and $\delta_\omega\in\mathbb{R}$ such that 
$$|D\phi_\omega(x,y)|^s=((x+\gamma_\omega)^2+(y+\delta_\omega)^2)^{-s}.$$
If we write $u=x+\gamma_\omega\geq \gamma+\sigma>0$ and $v=y+\delta_\omega$, then
$$\frac{\partial_{x}^{(j)}(|D\phi_\omega(x,y)|^s)}{|D\phi_\omega(x,y)|^s}=\frac{\partial_{u}^{(j)}((x+\gamma_\omega)^2+(y+\delta_\omega)^2)^{-s}}{((x+\gamma_\omega)^2+(y+\delta_\omega)^2)^{-s}}=\frac{\partial_{x}^{(j)}G(u,v;s)}{G(u,v;s)}.$$
Combining this with the bound (\ref{eq:bound_for_poly}) and passing through the limit using (\ref{eq:fs_deriv}), we obtain the bound in the theorem statement when $i=x$. The case for $i=y$ proceeds by a similar argument.
\end{proof}
\begin{remark}
For our numerical results, $\hat{X}$ will be a square. Its interior $H$ is therefore an open, bounded set. $H$ is also mildly regular because for any $x,y\in H$, by convexity we have that $\psi(t)=y+(1-t)(x-y)$ is a Lipschitz map $[0,1]\to H$ with $\psi(0)=x$, $\psi(1)=y$, and 
$$\int_0^1\|\psi'(t)\|\,dt=\int_0^1\|y-x\|\,dt=\|x-y\|.$$
\end{remark}
 %The following result of Falk and Nussbaum—which we also state in less than full generality for simplicity—provides most of the bounds we need for our discussion of two-dimensional continued fractions in Section 4 \cite{falk2018new}. 
%\begin{theorem}
%Let $E\subset\mathbb{N}\times\mathbb{Z}$ be finite. Let $f_s$ be the eigenfunction from the Perron-Frobenius operator (\ref{eq:perron}). Then
%\begin{align}
%-2s&\leq\frac{\partial_xf_s(x,y)}{f_s(x,y)}\leq0,\\
%-2s(2s+1)(2s+2)&\leq\frac{\partial^{3}_xf_s(x,y)}{f_s(x,y)}\leq\frac{2s(2s+2)}{(s+2)^2},\\
%\left|\frac{\partial_yf_s(x,y)}{f_s(x,y)}\right|&\leq s,\\
%\left|\frac{\partial^3_yf_s(x,y)}{f_s(x,y)}\right|&\leq 2s(2s+2)\max\{25\sqrt5/72,(2s+1)/8\}.
%\end{align}
%\label{thm:falk_deriv_bounds}
%\end{theorem}

\subsection{Hidden Positivity}

Previous work in the rigorous estimation of Hausdorff dimensions of conformal fractals has relied upon the following result (see Lemma 3.2 in \cite{falk2018c}).

\begin{lemma}
Let $M$ be an $(n+1)\times(n+1)$ matrix with nonnegative entries and $w$ be a vector with $n+1$ strictly positive components. Then
\begin{align*}
\text{if } (Mw)_j\geq\lambda w_j,&\text{ then } r(M)\geq\lambda,\\
\text{if } (Mw)_j\leq\lambda w_j,&\text{ then } r(M)\leq\lambda,
\end{align*}
where $r(M)$ denotes the spectral radius of $M$.
\label{lem:pos_mat}
\end{lemma}

The hypothesis that $M$ have nonnegative entries is indispensable. As such, an ideal approximation scheme ought to preserve positivity in the sense that if $f$ is a nonnegative function, its approximation $Qf$ also ought to be nonnegative.

Unfortunately, there are computational drawbacks to demanding preservation of positivity. A consequence of Korovkin's Theorem is that any linear approximation scheme that preserves positivity as described above is limited to be at most second-order accurate \cite{korovkin1957order}. Therefore, in order to achieve higher accuracy, one must either sacrifice positivity or turn to nonlinear methods. Here, we take the former approach.

In \cite{falk2021hidden}, Falk and Nussbaum showed that, by considering an alternative cone of vectors in place of the cone of nonnegative vectors, the conclusions of the Perron-Frobenius theorem can hold under the more general notion of positivity in which an operator maps a cone of vectors into itself. In the case of continuous piecewise polynomial functions, they defined a cone of functions $K_M$ such that $f\in K_M$ if and only if $f(x_i)\leq\exp(M|x_i-x_j|)f(x_j)$ for all $i,j$, where the $x_i$ are the interpolation points.

By considering this alternative cone, they were able to utilize the following result, which is an extension of the well-known theorem for positive matrices.

\begin{lemma}
Suppose $L(K_M\setminus\{0\})\subset K_{M'}\setminus\{0\}$ and $v\in K_M\setminus\{0\}$. If there exist positive constants $\alpha$ and $\beta$ such that
$$\alpha v\leq_K Lv\leq_K\beta v,$$
then $\alpha\leq r(L)\leq \beta$, where $r(L)$ is the spectral radius of $L$.
\label{lem:cones}
\end{lemma}
 
\section{Hidden Positivity for B-Splines}
In this section, we prove hidden positivity results for B-spline quasi-interpolants similar to those for polynomial interpolants by Falk and Nussbaum \cite{falk2021hidden}. These results are necessary for us to apply Lemma \ref{lem:cones} to obtain rigorous bounds for the Hausdorff dimensions of various continued fraction limit sets.

We'll comment here that we have separated the proofs for the one-dimensional case from those of the higher-dimensional case, despite the fact that the proofs for the higher-dimensional case also hold for the one-dimensional case and proceed nearly identically to the arguments presented in the one-dimensional case. We do this intentionally for two reasons. First, the one-dimensional case is the simplest, and we expect it will aid the reader in understanding the proof of the higher-dimensional case. Second, we hope this presentation emphasizes how results using B-splines can generalize to higher dimensions without much alteration to the one-dimensional case.

\subsection{One-Dimensional Case}

We'll begin by defining the cone of functions for B-splines. Let $n$ denote the degree of the B-splines, $J$ denote the number of subintervals into which $[0,1]$ will be partitioned, and $h=1/J$ denote the mesh size. Let $\hat{X}=[-nh,1+nh]$ and equip $\hat{X}$ with a uniform knot sequence $\{\xi_j\}_{j=0}^{J+2n}$ where $\xi_j=(j-n)h$. Let $\overline{\xi}_j=(\xi_j-\xi_{j-1})/2$ denote the midpoint of the $j$-th subinterval. In total, there are $N=J+2n$ subintervals. We can represent values $f(\overline\xi_1),f(\overline\xi_2),\dots,f(\overline\xi_N)$ of a function $f$ at these points by a vector $\mathbf{f}=(\mathbf{f}_1,\mathbf{f}_2,\dots,\mathbf{f}_N)\in\mathbb{R}^N$ where $\mathbf{f}_j=f(\overline\xi_j)$.

For $M>0$, define the cone $K_M\subset\mathbb{R}^N$ such that $\mathbf{f}\in K^M$ if and only if $$\mathbf{f}_i\leq \mathbf{f}_j\exp(M|\overline\xi_i-\overline\xi_j|)$$ for all $1\leq i,j\leq N$. Equivalently, if $\mathbf{f}\in K_M$ is not identically zero, then $\mathbf{f}_j>0$ for all $j$ and $$|\log\mathbf{f}_i-\log\mathbf{f}_j|\leq M|\overline\xi_i-\overline\xi_j|$$ for all $1\leq i,j\leq N$.

For $\mathbf{f}\in\mathbb{R}^N$, define the function $\mathbf{Qf}(x):\hat{X}\to\mathbb{R}$ as
$$\mathbf{Qf}(x)=\sum_{j=1}^J(Q_j\mathbf{f})b_j(x),\quad\text{where}\quad Q_j\mathbf{f}=\sum_{v=0}^nw_v\mathbf{f}_{j+v}.$$

The weights $w_v$ satisfy $\sum_{v=0}^nw_v=1$ and we define $\|Q\|=\sum_{v=0}^n|w_v|$. Let $W^+=\sum\{w_v:w_v>0\}$ be the sum of the positive weights and $W^-=\sum\{w_v:w_v<0\}$ be the sum of the negative weights. Then $W^++W^-=1$ and $W^+-W^-=\|Q\|$.

Note that if $f$ is a function defined on $\hat{X}$ whose values at the points $\overline{\xi}_j$ are encoded by the vector $\mathbf{f}\in\mathbb{R}^N$, then $Qf(x)$ and $\mathbf{Qf}(x)$ are the same function, and thus any approximation properties of $Qf$ described in Section 2.1 also hold for $\mathbf{Qf}$. 

For $s>0$ and $\mathbf{f}\in\mathbb{R}^N$, define the linear operator $\mathbf{L}_s:\hat{X}\to\mathbb{R}^N$ by
$$\mathbf{L}_s(\mathbf{f})(x)=\sum_{e\in E}\|D\phi_e(\overline\xi_j)\|^s\mathbf{Qf}(\phi_e(x)).$$
We can obtain a finite-dimensional approximation of (\ref{eq:perron}) by substituting the values $\overline{\xi}_j$ for $x$, which results in an $N\times N$ matrix given by
\begin{equation}
\mathbf{L}_s(\mathbf{f})(\overline\xi_j)=\sum_{e\in E}\|D\phi_e(\overline\xi_j)\|^s\mathbf{Qf}(\phi_e(\overline\xi_j)).
\label{eq:approx_no_iter}
\end{equation}

We will also need to work with iterates of this operator, which we define as follows. If $v\geq1$ is the number of times the operator is to be iterated, then we define $E^v=E\times E\times\dots\times E$ ($v$ times) and 
\begin{equation}
\mathbf{L}_{s,v}(\mathbf{f})(\overline\xi_j)=\sum_{e\in E^v}\|D\phi_e(\overline\xi_j)\|^s\mathbf{Qf}(\phi_e(\overline\xi_j)).
\label{eq:approx_no_iter}
\end{equation}

Our goal is to use $\mathbf{L}_{s,v}$ to approximate the spectral radius of the Perron-Frobenius operator (\ref{eq:perron}) in a rigorous manner. In particular we wish to apply Lemma \ref{lem:cones}. In a series of multiple steps, we will show that, when $v$ is sufficiently large and $h$ is sufficiently small, the hypotheses of Lemma \ref{lem:cones} are satisfied. 

We begin by showing that, so long as $h$ is sufficiently small, the quasi-interpolant of a function $\mathbf{f}\in K_M$ is strictly positive.

\begin{theorem}
Let $\mathbf{f}\in K_M$ be defined as above, and for $n\geq1$ define the quantities
$$E(h,n)=\begin{cases}
\exp(Mhn/2)W^++\exp(-Mhn/2)W^-,&n\ \mathrm{even}\\
\exp(Mh(n+1)/2)W^++\exp(-Mh(n+1)/2)W^-,&n\ \mathrm{odd}
\end{cases}
$$
and 
$$
F(h,n)=\begin{cases}
\exp(-Mhn/2)W^++\exp(Mhn/2)W^-,&n\ \mathrm{even}\\
\exp(-Mh(n+1)/2)W^++\exp(Mh(n+1)/2)W^-,&n\ \mathrm{odd}
\end{cases}
$$
Define the function $\mathbf{Sf}:\hat{X}\to\mathbb{R}$ by
$$\mathbf{Sf}(x)=\sum_{j=1}^J\mathbf{f}_{j+\lfloor n/2\rfloor}b_j(x).$$
Then $\mathbf{Sf}(x)>0$ for all $x$ and 
$$F(h,n)\mathbf{Sf}(x)\leq \mathbf{Qf}(x)\leq E(h,n)\mathbf{Sf}(x)$$
for all $x\in\hat{X}$. Moreover, if 
\begin{equation}h<\begin{cases}\dfrac{\log|W^+/W^-|}{Mn},&n\ \mathrm{even}\\ &\\
\dfrac{\log|W^+/W^-|}{M(n+1)},&n\ \mathrm{odd}
\end{cases},
\label{eq:h_cond_pos}
\end{equation}
then $\mathbf{Qf}(x)>0$ for all $x\in\hat{X}$.
\label{thm:1dhiddenpos}
\end{theorem}
\begin{proof}
For each $j=1,\dots, J$, we have
\begin{align*}
Q_j\mathbf{f}&=\sum_{w_v>0}w_v\mathbf{f}_{j+v}+\sum_{w_v<0}w_v\mathbf{f}_{j+v}
\end{align*}
If $n$ is even, then for each $v$, $\overline{\xi}_{j+v}$ is at a distance at most $hn/2$ from $\overline{\xi}_{j+n/2}$. By definition of $K_M$, for any $v$ we have
$$\mathbf{f}_{j+n/2}\exp(-Mhn/2)\leq \mathbf{f}_{j+v}\leq \mathbf{f}_{j+n/2}\exp(Mhn/2)$$
Thus, we can bound $Q_j\mathbf{f}$ above and below by
$$\exp(-Mhn/2)\sum_{w_v>0}w_v\mathbf{f}_{j+n/2}+\exp(Mhn/2)\sum_{w_v<0}w_v\mathbf{f}_{j+n/2}$$
$$\leq Q_j\mathbf{f}\leq$$
$$\exp(Mhn/2)\sum_{w_v>0}w_v\mathbf{f}_{j+n/2}+\exp(-Mhn/2)\sum_{w_v<0}w_v\mathbf{f}_{j+n/2}=.$$
Thus for each $j$
$$F(h,n)\mathbf{f}_{j+n/2}\leq Q_j\mathbf{f}\leq E(h,n)\mathbf{f}_{j+n/2}$$
We thus have, noting $\lfloor n/2\rfloor=n/2$ for even $n$,
$$\mathbf{Qf}(x)=\sum_{j=1}^J(Q_j\mathbf{f})b_j(x)\leq \sum_{j=1}^JE(h,n)\mathbf{f}_{j+n/2}b_j(x)=E(h,n)\mathbf{Sf}(x)$$
A similar calculation shows $F(h,n)\mathbf{Sf}(x)\leq \mathbf{Qf}(x)$. Notice $\mathbf{Sf}(x)>0$ since $\mathbf{f}_{j+n/2}>0$ for all $j=1,\dots, J$ and $b_j(x)>0$ for all $j=1,\dots, J$ and $x\in\hat{X}$. 

We have $\mathbf{Qf}(x)>0$ so long as $F(h,n)>0$, which, after a straightforward calculation, occurs when
$$h<\frac{\log|W^+/W^-|}{Mn}.$$
\end{proof}

If $n$ is odd, then for each $v$, $\overline{\xi}_{j+v}$ is at a distance at most $h(n+1)/2$ from $\overline{\xi}_{j+\lfloor n/2\rfloor}$. Repeating the arguments above with $\mathbf{f}_{j+n/2}$ replaced by $\mathbf{f}_{j+\lfloor n/2\rfloor}$ and $\exp(\pm Mhn/2)$ replaced by $\exp(\pm Mh(n+1)/2)$ proves the result for the odd $n$ case. Again, a calculation shows that $F(h,n)>0$ when 
$$h<\frac{\log|W^+/W^-|}{M(n+1)}.$$

\begin{theorem}
Let $\mathbf{f}\in K_M$ and $h>0$ be small enough such that (\ref{eq:h_cond_pos}) is satisfied. Then for all $x,y$ in the parameter interval of $\mathbf{Qf}$ we have
$$|\log \mathbf{Qf}(x)-\log \mathbf{Qf}(y)|\leq C|x-y|$$
where, if $n$ is even, 
$$C=\frac{(E(h,n)-F(h,n)\exp(-Mh))\exp(Mhn)}{hF(h,n)}$$
and $C\to M(1+n\|Q\|)$ as $h\to0^+$. If $n$ is odd, then 
$$C=\frac{(E(h,n)-F(h,n)\exp(-Mh))\exp(Mh(n+1))}{hF(h,n)}$$
and $C\to M(1+(n+1)\|Q\|)$ as $h\to0^+$.
\label{thm:1d_lip}
\end{theorem}
\begin{proof}
We have, for $x,y\in X$ (which coincides with the parameter interval of $\mathbf{Qf}$, where without loss of generality we may take $x<y$,
$$|\log \mathbf{Qf}(x)-\log \mathbf{Qf}(y)|=\left|\int_x^y\frac{(\mathbf{Qf})'(t)}{\mathbf{Qf}(t)}\,dt\right|\leq |x-y|\cdot\max_{t\in[x,y]}\left|\frac{(\mathbf{Qf})'(t)}{\mathbf{Qf}(t)}\right|.$$
Using the differentiation formula for B-splines, and noting that $0<F(h,n)\mathbf{f}_{j+\lfloor n/2\rfloor}\leq Q_j\mathbf{f}\leq E(h,n)\mathbf{f}_{j+\lfloor n/2\rfloor} $ for all $j$ since $h$ is sufficiently small,
\begin{align*}
    (\mathbf{Qf})'(t)&=\sum_{j=2}^J\frac{1}{h}[Q_j\mathbf{f}-Q_{j-1}\mathbf{f}]b_j^{n-1}(t)\\
    &\leq \sum_{j=2}^J\frac{1}{h}[E(h,n)\mathbf{f}_{j+\lfloor n/2\rfloor}-F(h,n)\mathbf{f}_{j-1+\lfloor n/2\rfloor}]b_j^{n-1}(t)\\
    &\leq \sum_{j=2}^J\frac{1}{h}[E(h,n)-F(h,n)\exp(-Mh)]\mathbf{f}_{j+\lfloor n/2\rfloor}b_j^{n-1}(t).
\end{align*}
Bounding $\mathbf{Qf}(x)$ below by $F(h,n)\mathbf{Sf}(x)\leq \mathbf{Qf}(x)$, we have
$$\left|\frac{(\mathbf{Qf})'(t)}{\mathbf{Qf}(t)}\right|\leq \frac{E(h,n)-F(h,n)\exp(-Mh)}{hF(h,n)}\cdot \frac{\sum_{j=2}^J\mathbf{f}_{j+\lfloor n/2\rfloor}b^{n-1}_j(t)}{\sum_{j=1}^J\mathbf{f}_{j+\lfloor n/2\rfloor}b_j^n(t)}.$$
Let $J_n$ and $J_{n-1}$ denote the set of indices such that $b_j^n(t)>0$ and $b_j^{n-1}(t)>0$, respectively. If $n$ is even, then $b_j^n$ spans $n+1$ knot intervals and thus contains $n+1$ elements. Let $j^*$ denote the median of $J_n$, then%If $n$ is even, let $j^*$ denote the median of $J_n$, otherwise, let $j^*$ denote the median of $J_{n-1}. Then for all $j\in J_n$ and $j\in J_{n-1}$, $j^*+{n/2}$
$$\frac{\sum_{j=2}^J\mathbf{f}_{j+n/2}b^{n-1}_j(t)}{\sum_{j=1}^J\mathbf{f}_{j+n/2}b_j^n(t)}\leq \frac{\exp(Mhn/2)\sum_{j\in J_{n-1}}\mathbf{f}_{j^*+n/2}b_j^{n-1}(t)}{\exp(-Mhn/2)\sum_{j\in J_{n}}\mathbf{f}_{j^*+n/2}b_j^{n}(t)}=\exp(Mhn)\frac{\sum_{j\in J_{n-1}}b_j^{n-1}(t)}{\sum_{j\in J_{n}}b_j^{n}(t)}.$$
Since $t\in[x,y]$ is contained in the parameter interval, $\sum_{j\in J_{n-1}}b_j^{n-1}(t)=\sum_{j\in J_{n}}b_j^{n}(t)=1$, so we obtain
$$\left|\frac{(\mathbf{Qf})'(t)}{\mathbf{Qf}(t)}\right|\leq \frac{[E(h,n)-F(h,n)\exp(-Mh)]\exp(Mhn)}{hF(h,n)}:=C.$$
Thus $|\log Qf(x)-\log Qf(y)|\leq C|x-y|$ for $x,y$ in the parameter interval.

As $h\to0$, $E(h,n)$ and $F(h,n)\to W^++W^-=1$, so $C\to \frac00$ as $h\to0$. Thus by L'Hopital's rule,
$$\lim_{h\to0}C=\frac{A(h,n)+B(h,n)}{F(h,n)+hF'(h,n)}$$
where
$$A(h,n):=Mn[E(h,n)-F(h,n)\exp(-Mh)]\exp(Mhn)$$
and 
$$B(h,n):=[E'(h,n)-F'(h,n)\exp(Mh)+MF(h,n)\exp(-Mh)]\exp(Mhn),$$
where the $'$ notation means a derivative in $h$.
As $h\to0$, $A(h,n)\to0$, and for $B(h,n)$ we have 
$$E'(h,n)=(Mn/2)\exp(Mhn/2)W^+-(Mn/2)\exp(-Mhn/2)W^-\xrightarrow[h\to0]{}(Mn/2)(W^+-W^-)=\|Q\|Mn/2$$
and
$$F'(h,n)=(-Mn/2)\exp(-Mhn/2)W^++(Mn/2)\exp(Mhn/2)W^-\xrightarrow[h\to0]{}(Mn/2)(W^--W^+)=-\|Q\|Mn/2.$$
Therefore $B(h,n)\to \|Q\|Mn+M$ as $h\to0$. The denominator $F(h,n)+hF'(h,n)\to1$, so it follows that $C\to M(1+\|Q\|n)$.

If $n$ is odd, then $b_j^{n-1}$ spans $n$ knot intervals, so $J_{n-1}$ has an odd number of elements. It then follows that $\overline{\xi}_{j+\lfloor n/2}$ is at a distance at most $h(n+1)/2$ from $\overline{\xi}_{j^*+\lfloor n/2}$ for all $j\in J_n$, so
$$\frac{\sum_{j=2}^J\mathbf{f}_{j+\lfloor n/2\rfloor}b^{n-1}_j(t)}{\sum_{j=1}^J\mathbf{f}_{j+\lfloor n/2\rfloor}b_j^n(t)}\leq \frac{\exp(Mh(n+1)/2)\sum_{j\in J_{n-1}}\mathbf{f}_{j^*+n/2}b_j^{n-1}(t)}{\exp(-Mh(n+1)/2)\sum_{j\in J_{n}}\mathbf{f}_{j^*+n/2}b_j^{n}(t)}=\exp(Mh(n+1))\frac{\sum_{j\in J_{n-1}}b_j^{n-1}(t)}{\sum_{j\in J_{n}}b_j^{n}(t)}.$$
As before, since $t$ is contained in the parameter interval, we obtain
$$\left|\frac{(\mathbf{Qf})'(t)}{\mathbf{Qf}(t)}\right|\leq \frac{[E(h,n)-F(h,n)\exp(-Mh)]\exp(Mh(n+1))}{hF(h,n)}:=C,$$
and the limit as $h\to0^+$ proceeds by a similar argument as in the even case.
\end{proof}

Our next step is to show that, for any pair of real numbers $M$ and $M'$ with $0<M'<M$, there is some number of iterations $v$ such that $\mathbf{L}_{s,v}(K_M\setminus\{0\})\subset K_{M'}$. Before we prove this result, we need the following remark.

\begin{remark}
For any $v\geq1$, using the notation in \cite{falk2021hidden}, there is some constant $M_0(v)$ such that $|\log\|D\phi_e(x)\|-\log\|D\phi_e(y)\||\leq M_0(v)|x-y|$ for all $x,y\in \hat{X}$ and $e\in E^v$. This follows from the fact that, for each $e\in E^v$, $D\phi_e$ is a differentiable map that is bounded away from 0, so $\log\|D\phi_e\|$ is Lipschitz, as well.
\end{remark}

\begin{theorem}
Let $L(v)$ denote the maximal Lipschitz constant among $\{\phi_e\}_{e\in E^v}$. Let $C$ be defined as in the statement of Theorem \ref{thm:1d_lip} and take $h>0$ small enough such that (\ref{eq:h_cond_pos}) is satisfied. If $v$ is taken large enough so that 
\begin{equation}
C\cdot L(v)\leq M'-M_0(v)s,
\label{eq:cone_map_cond}
\end{equation}
then $\mathbf{L}_{s,v}(K_M\setminus\{0\})\subset K_{M'}$.
\label{thm:1d_cone_mapping}
\end{theorem}
\begin{proof}
Because cones are closed under addition, it suffices to prove that the map $\overline{\xi}_j\mapsto \|D\phi_e(\overline{\xi}_j)\|^s\mathbf{Qf}(\phi_e(\overline{\xi_j}))$ belongs to $K_{M'}$ for all $e\in E^v$. By the remark above, we have for all $x,y\in \hat{X}$,
$$|\log\|D\phi_e(x)\|^s-\log\|D\phi_e(y)\|^s|\leq M_0(v)s|x-y|.$$
For $v$ sufficiently large, we have that $\phi_e(\hat{X})\subset X$, and thus $\phi_e(x)$ and $\phi_e(y)$ belong to the parameter interval of $\mathbf{Qf}$. By Theorem $\ref{thm:1d_lip}$, we have
$$|\log \mathbf{Qf}(\phi_e(x))-\log \mathbf{Qf}(\phi_e(y))|\leq C|\phi_e(x)-\phi_e(y)|\leq C\cdot L(v)|x-y|.$$
Combining these inequalities, we obtain
\begin{align*}
|\log(\|D\phi_e(x)\|^s\mathbf{Qf}(\phi_e(x)))-\log(\|D\phi_e(y)\|^s\mathbf{Qf}(\phi_e(y)))|&\leq|\log\|D\phi_e(x)\|^s-\log\|D\phi_e(y)\|^s|\\
&\quad\quad\quad+|\log \mathbf{Qf}(\phi_e(x))-\log \mathbf{Qf}(\phi_e(y))|\\
&\leq (M_0(v)s+C\cdot L(v))|x-y|\\
&\leq M'|x-y|.
\end{align*}
Thus $\overline{\xi}_j\mapsto \|D\phi_e(\overline{\xi}_j)\|^s\mathbf{Qf}(\phi_e(\overline{\xi_j}))$ belongs to $K_{M'}$ and so $\mathbf{L}_{s,v}(K_M\setminus\{0\})\subset K_{M'}$.
\end{proof}

The final step in order to apply Lemma \ref{lem:cones} is to show that, for some vector $\mathbf{f}\in K_M\setminus\{0\}$, there are constants $\alpha$ and $\beta$ such that $$\alpha \mathbf{f}\leq_K\mathbf{L}_{s,v}(\mathbf{f})\leq_K\beta\mathbf{f},$$
where $\mathbf{f}\leq_K\mathbf{g}$ if and only if $\mathbf{g}-\mathbf{f}\in K_M$. 

\begin{theorem}
Let $f_s$ be the strictly positive eigenfunction associated to the Perron-Frobenius operator defined by $(\ref{eq:perron})$ with eigenvalue $\lambda_s$, and let $\mathbf{f_s}\in\mathbb{R}^N$ be given by $(\mathbf{f_s})_j=f_s(\overline{\xi}_j)$ for $j=1,\dots,N$ with $N=J+2n$. Let $A_k=\max_{x\in \hat{X}}|f^{(k)}_s/f_s|$ for $k\geq1$, take $h>0$ small enough such that (\ref{eq:h_cond_pos}) is satisfied, and define
$$B=\frac{(n+1)^n\|Q\|A_{n+1}e^{A_1(2n+1)h}}{n!}.$$
Then for all $x\in \hat{X}$,
\begin{equation}
(1-Bh^{n+1})\lambda_s^vf_s(x)\leq \mathbf{L}_{s,v}\mathbf{f_s}(x)\leq (1+Bh^{n+1})\lambda_s^vf_s(x)
\label{eq:1d_setup_1}
\end{equation}
Let $M=A_1+2/(1-h^2)$ then with $K=K_M$,
\begin{equation}
(1-Bh^{n})\lambda_s^v\mathbf{f_s}\leq_K \mathbf{L}_{s,v}\mathbf{f_s}\leq_K (1+Bh^{n})\lambda_s^v\mathbf{f_s}.
\label{eq:1d_setup_2}
\end{equation}
%For $M'$ such that $0<M'<M$, take $v\geq1$ large enough so that (\ref{eq:cone_map_cond}) is satisfied, then Lemma \ref{lem:cones} is 
\label{thm:1d_lemma_setup}
\end{theorem}
\begin{proof}
We begin by noting that since $|f_s'/f_s|\leq A_1$, it follows for all $x,y\in \hat{X}$ that
$$|\log f_s(x)-\log f_s(y)|\leq \int_{x}^y\max_{t\in[x,y]}\left|\frac{f'_s(t)}{f_s(t)}\right|\leq A_1|x-y|.$$
Thus $f_s\in K_{A_1}$ and in particular, $f_s(x)\leq f_s(y)e^{A_1|x-y|}$ for all $x,y\in \hat{X}$.

We will prove the rightmost inequalities of (\ref{eq:1d_setup_1}) and (\ref{eq:1d_setup_2}), since the leftmost inequalities will follow by analogous arguments. To prove (\ref{eq:1d_setup_1}), we have
$$\mathbf{L}_{s,v}\mathbf{f_s}(x)=\sum_{e\in E^v}\|D\phi_e(x)\|^s\mathbf{Qf_s}(\phi_e(x))=\sum_{e\in E^v}\|D\phi_e(x)\|^sQf_s(\phi_e(x)),$$
since $\mathbf{Qf}$ and $Qf$ are the same function. By Theorem \ref{thm:quasi_error}, it follows that 
$$Qf_s(\phi_e(x))\leq |Qf_s(\phi_e(x))-f_s(\phi_e(x))|+|f_s(\phi_e(x))|\leq f_s(\phi_e(x))+\frac{(n+1)^n\|Q\|}{n!}h^{n+1}\|f_s^{(n+1)}\|_{\infty,D_{\phi_e(x)}},$$
where $\|f_s^{(n+1)}\|_{\infty,D_{\phi_e(x)}}$ is the maximum of $|f_s^{(n+1)}|$ over $D_{\phi_e(x)}$, the union of the supports of the splines containing $\phi_e(x)$, which is an interval of length $(2n+1)h$. Let the maximum be achieved at $x^*\in D_{\phi_e(x)}$, then since $|f^{(n+1)}_s(x^*)/f_s(x^*)|\leq A_{n+1}$ and $f_s\in K_{A_1}$, it follows that 
$$|f^{(n+1)}_s(x^*)|\leq A_{n+1}f_s(x^*)\leq A_{n+1}e^{A_1(2n+1)h}f(\phi_e(x)).$$
We therefore have
\begin{align*}
\mathbf{L}_{s,v}\mathbf{f_s}(x)&\leq\sum_{e\in E^v}\|D\phi_e(x)\|^s\left[f_s(\phi_e(x))+\frac{(n+1)^n\|Q\|A_{n+1}e^{A_1(2n+1)h}}{n!}h^{n+1}f_s(\phi_e(x))\right]\\
&=(1+Bh^{n+1})\sum_{e\in E^v}\|D\phi_e(x)\|^sf_s(\phi_e(x))\\
&=(1+Bh^{n+1})L_{s,v}f_s(x)=(1+Bh^{n+1})\lambda_s^vf_s(x).
\end{align*}
The leftmost inequality follows a similar proof, which we omit for brevity. We assume now that (\ref{eq:1d_setup_1}) holds. To prove (\ref{eq:1d_setup_2}), we need to show that $(1+Bh^{n})\lambda_s^v\mathbf{f_s}-\mathbf{L}_{s,v}\mathbf{f_s}\in K_M$, with $M=A_1+2/(1-h^2)$. From (\ref{eq:1d_setup_1}), for all $x\in \hat{X}$ we have %For any pair $1\leq i\neq j\leq N$
$$(1+Bh^{n})\lambda_s^vf_s(x)-\mathbf{L}_{s,v}\mathbf{f_s}(x)\leq(1+Bh^n)\lambda_s^vf_s(x)-(1+Bh^{n+1})\lambda_s^vf_s(x)$$
and for all $y\in\hat{X}$,
$$(1+Bh^{n})\lambda_s^vf_s(y)-\mathbf{L}_{s,v}\mathbf{f_s}(y)\geq(1+Bh^n)\lambda_s^vf_s(y)-(1-Bh^{n+1})\lambda_s^vf_s(y).$$
Therefore, for any pair of indices $1\leq i\neq j\leq N$, we can set $x=\overline{\xi}_i$ and $y=\overline{\xi}_j$ to obtain
$$(1+Bh^{n})\lambda_s^v(\mathbf{f_s})_i-(\mathbf{L}_{s,v}\mathbf{f_s})_i\leq (1+h)Bh^n\lambda_s^v(\mathbf{f_s})_i$$
and
$$(1+Bh^{n})\lambda_s^v(\mathbf{f_s})_j-(\mathbf{L}_{s,v}\mathbf{f_s})_j\geq (1-h)Bh^n\lambda_s^v(\mathbf{f_s})_j.$$
Therefore
$$\frac{(1+Bh^{n})\lambda_s^v(\mathbf{f_s})_i-(\mathbf{L}_{s,v}\mathbf{f_s})_i}{(1+Bh^{n})\lambda_s^v(\mathbf{f_s})_j-(\mathbf{L}_{s,v}\mathbf{f_s})_j}\leq\frac{1+h}{1-h}\frac{(\mathbf{f_s})_i}{(\mathbf{f_s})_j}\leq\exp\left(\log\left|\frac{1+h}{1-h}\right|+A_1|\overline{\xi}_i-\overline{\xi}_j|\right).$$
We can bound $\log\left|\frac{1+h}{1-h}\right|$ above using the trapezoid rule, since $t\mapsto 1/t$ is convex, 
\begin{align*}
|\log(1+h)-\log(1-h)|&=\int_{1-h}^{1+h}\frac{1}{t}\,dt\\
&\leq2h\cdot\frac{1}{2}\left(\frac{1}{1+h}+\frac{1}{1-h}\right)=\frac{2h}{1-h^2}.
\end{align*}
We therefore have, using the fact that $|\overline{\xi}_i-\overline{\xi}_j|\geq h$ for all $i,j$ with $i\neq j$,
$$\log\left|\frac{1+h}{1-h}\right|\leq\frac{2h}{1-h^2}\frac{|\overline{\xi}_i-\overline{\xi}_j|}{|\overline{\xi}_i-\overline{\xi}_j|}\leq\frac{2}{1-h^2}|\overline{\xi}_i-\overline{\xi}_j|.$$
Thus
$$\frac{(1+Bh^{n})\lambda_s^v(\mathbf{f_s})_i-(\mathbf{L}_{s,v}\mathbf{f_s})_i}{(1+Bh^{n})\lambda_s^v(\mathbf{f_s})_j-(\mathbf{L}_{s,v}\mathbf{f_s})_j}\leq\frac{1+h}{1-h}\frac{(\mathbf{f_s})_i}{(\mathbf{f_s})_j}\leq\exp\left((2(1-h^2)^{-1}+A_1)|\overline{\xi}_i-\overline{\xi}_j|\right),$$
so $\mathbf{L}_{s,v}\mathbf{f_s}\leq_K (1+Bh^{n})\lambda_s^v\mathbf{f_s}$.
\end{proof}

Before we move on to the higher dimensional case, a natural question arises: what is the minimum number of iterations $v$ needed in order for the results above to hold? Ideally, one would like to keep the number of iterations small, as each iteration increases the number of computations by a factor equal to the number of elements in the alphabet. 

From Theorem \ref{thm:1d_cone_mapping}, we require $L(v)\cdot C\leq M'-M_0(v)s$. Since $C$ approaches a constant multiple of $M$ as $h\to0$, and we can select $M'$ as close to $M$ as we like, so long as $M'<M$, it follows that the necessary number of iterations is the minimum number $v$ such that $L(v)\cdot C/M<1$. Under this condition, we can select $0<M'<M$ sufficiently large so that the hypotheses of Theorem \ref{thm:1d_cone_mapping} are satisfied. This, in turn, will require $h$ to be smaller, but as we will see in the numerical examples, the conditions on $h$ produced by these choices will be easily satisfied.

%With $M$ given as in Theorem \ref{thm:1d_lemma_setup}, the goal is to set $M'$, $n$, $h$, and $v$ appropriately so that the hypotheses of Theorems \ref{thm:1d_lip} and \ref{thm:1d_lemma_setup} are satisfied, and thus Lemma \ref{lem:cones} applies. The most important consideration for computations is to keep $v$ as small as possible, ideally taking $v=1$. Since all of the constants in the results above are computable, we can determine \textit{a priori} the number of iterations needed, depending on the minimal element of alphabet $E$. 

%If $\min E=1$, then iteration of the alphabet is necessary in any case, since $\phi_1$ has Lipschitz constant 1

%Fortunately, if $\min E\geq3$, then only one iteration of the operator is necessary. Suppose $\min E=k\geq3$, then $L(1)=1/k^2$ since that is the Lipschitz constant of $\phi_k$. We also have
%$$|\log|\phi'_k(x)|-\log|\phi'_k(y)|\leq\int_x^y\left|\frac{\phi''_k(t)}{\phi'_k(t)}\right|\,dt\leq \max_{t\in[x,y]}\frac{2}{(t+k)}|x-y|=\frac{2}{k}|x-y|$$
%Thus $M_0(1)=2/k$.

\subsection{Higher-Dimensional Case}

Again, we'll begin by describing the cone of functions for B-splines on $\mathbb{R}^d$. Let $X=[0,1]\times[-\frac12,\frac12]^{d-1}$. Let $n$ denote the degree of the B-splines (we will use the same degree for all axes), $J$ denote the number of subintervals into which each axis of $X$ will be partitioned, and $h=1/J$ denote the mesh size on each axis. Let $\hat{X}=[-nh,1+nh]\times[-\frac12-nh,\frac12+nh]^{d-1}$. For the knot sequences, we will denote the $j$-th knot on the $k$-th coordinate axis by $\xi_{j_k}$, $j=0,1,\dots,J+2n$, $k=1,2\dots,d$. Equip the first coordinate axis of $\hat{X}$ with the uniform knot sequence $\{\xi_{j_1}\}_{j_1=0}^{J+2n}$ where $\xi_{j_1}=(j_1-n)h$, and let $\overline{\xi}_{j_1}=(\xi_{j_1}-\xi_{j_1-1})/2$ denote the midpoint of the $j_1$-th subinterval. Equip every other coordinate axis $k=2,3,\dots,d$ of $\hat{X}$ with the uniform knot sequence $\{\xi_{j_k}\}_{j_k=0}^{J+2n}$ where $\xi_{j_k}=-\frac12+(j_k-n)h$, and similarly let $\overline{\xi}_{j_k}=(\xi_{j_k}-\xi_{j_k-1})/2$ denote the midpoint of the $j_k$-th subinterval. In total, there are $N=J+2n$ subintervals on each axis.

The knot sequences on the axes partition the hypercube $\hat X$ into $N^d$ hypercubes of side length $h$. We will denote the midpoint of any such hypercube as $\overline{\xi}_j=(\overline{\xi}_{j_1},\overline{\xi}_{j_2},\dots,\overline{\xi}_{j_d})$. Throughout this section, we will treat the index $j$ on $\overline{\xi}_j$ as a vector $j=(j_1,j_2,\dots,j_d)$, so if $v=(v_1,v_2,\dots, v_d)$ is another vector, we write $\overline{\xi}_{j+v}=(\overline{\xi}_{j_1+v_1},\overline{\xi}_{j_2+v_2},\dots,\overline{\xi}_{j_d+v_d})$. Let $\mathcal{J}=\{j=(j_1,j_2,\dots,j_d):1\leq j_1,j_2,\dots,j_d\leq J+2n\}$. Given a function $f:\hat{X}\to\mathbb{R}$, we can represent its values at the midpoints $\{f(\overline{\xi}_j)\}_{j\in\mathcal{J}}$ as a vector $\mathbf{f}\in\mathbb{R}^{N^d}$ with $f(\overline{\xi}_j)=\mathbf{f}_{\psi(j)}$, where $\psi:\{1,2,\dots,N\}^d\to\{1,2,\dots,N^d\}$ is the function
$$\psi(j)=\psi(j_1,j_2,\dots,j_d)=j_1+N(j_2-1)+N^2(j_3-1)+\dots+N^{d-1}(j_d-1).$$

\begin{remark}
To avoid having to repeatedly write $\psi$ in the index of $\mathbf{f}$, we will abuse notation and simply write $\mathbf{f}_j$ in place of $\mathbf{f}_{\psi(j)}$, and any vector operations in the index will be performed before passing through $\psi$. That is, if $v=(v_1,v_2,\dots,v_d)$, we will write $\mathbf{f}_{j+v}=\mathbf{f}_{\psi(j+v)}=\mathbf{f}_{\psi(j_1+v_1,j_2+v_2,\dots,j_d+v_d)}$.
\end{remark}
%In words, $\psi$ takes the components of the vector $j$, subtracts 1 from each component, converts the $d$-digit, base $N$ number $

For $M>0$, define the cone $K_M\subset \mathbb{R}^{N^d}$ such that $\mathbf{f}\in K_M$ if and only if 
$$\mathbf{f}_i\leq\mathbf{j}\exp(M|\overline{\xi}_i-\overline{\xi}_j|)$$
for all $i,j\in\mathcal{J}$. As in the in one dimensional case, if $\mathbf{f}\in K_M$ is not identically zero, then $\mathbf{f}_j>0$ for all $j$ and 
$$|\log\mathbf{f}_i-\log\mathbf{f}_j|\leq M|\overline{\xi}_i-\overline{\xi}_j|$$
for all $i,j\in\mathcal{J}$.

For $\mathbf{f}\in\mathbb{R}^N$, define the function $\mathbf{Qf}(x):\hat{X}\to\mathbb{R}$ as
$$\mathbf{Qf}(x)=\sum_{j\in\mathcal{J}}(Q_j\mathbf{f})b_j(x),$$
where for $v=(v_1,v_2,\dots,v_d)$ we define
$$Q_j\mathbf{f}=\sum_{v_1=0}^n\sum_{v_2=0}^n\dots\sum_{v_d=0}^nw_{v_1}w_{v_2}\dots w_{v_d}\mathbf{f}_{j+v}$$
and
$$b_j(x)=b_{(j_1,j_2,\dots,j_d)}b(x_1,x_2,\dots,x_d)=\prod_{i=1}^db_{j_i}(x_i).$$

For $v=(v_1,v_2,\dots,v_d)$ we define the product weight as $W_v=w_{v_1}w_{v_2}\dots w_{v_d}$. As in the one dimensional case, these weights satisfy $\sum_{v}W_v=1$ and we define $\|Q\|=\sum_{v}|W_v|$. Let $W^+=\sum\{W_v:W_v>0\}$ be the sum of the positive weights and $W^-=\sum\{W_v:W_v<0\}$ be the sum of the negative weights. Then $W^++W^-=1$ and $W^+-W^-=\|Q\|$.

%Note that if $f$ is a function defined on $\hat{X}$ whose values at the points $\overline{\xi}_j$ are encoded by the vector $\mathbf{f}\in\mathbb{R}^N$, then $Qf(x)$ and $\mathbf{Qf}(x)$ are the same function, and thus any approximation properties of $Qf$ described in Section 2.1 also hold for $\mathbf{Qf}$. 

For $s>0$ and $\mathbf{f}\in\mathbb{R}^{N^d}$, we define the operator $\mathbf{L}_s:\mathbb{R}^{N^d}\to\mathbb{R}$ in a similar form to the one dimensional case,
$$\mathbf{L}_s(\mathbf{f})(x)=\sum_{e\in E}\|D\phi_e(\overline\xi_j)\|^s\mathbf{Qf}(x)).$$
Again, by substituting the points $\overline{\xi}_j$, $j\in\mathcal{J}$, for $x$, we obtain an $N^{d}\times N^d$ matrix given by
\begin{equation}
\mathbf{L}_s(\mathbf{f})(\overline\xi_j)=\sum_{e\in E}\|D\phi_e(\overline\xi_j)\|^s\mathbf{Qf}(\phi_e(\overline\xi_j)).
\label{eq:approx_no_iter_d_dim}
\end{equation}

We will also need to work with iterates of this operator, which we define as follows. If $v\geq1$ is the number of times the operator is to be iterated, then we define $E^v=E\times E\times\dots\times E$ ($v$ times) and 
\begin{equation}
\mathbf{L}_{s,v}(\mathbf{f})(\overline\xi_j)=\sum_{e\in E^v}\|D\phi_e(\overline\xi_j)\|^s\mathbf{Qf}(\phi_e(\overline\xi_j)).
\label{eq:approx_iter_d_dim}
\end{equation}

\begin{theorem}
Let $\mathbf{f}\in K_M$ be defined as above, and for $n\geq1$ define the quantities
$$E(h,n)=\begin{cases}
\exp(Mhn\sqrt{d}/2)W^++\exp(-Mhn\sqrt{d}/2)W^-,&n\ \mathrm{even}\\
\exp(Mh(n+1)\sqrt{d}/2)W^++\exp(-Mh(n+1)\sqrt{d}/2)W^-,&n\ \mathrm{odd}
\end{cases}
$$
and 
$$
F(h,n)=\begin{cases}
\exp(-Mhn\sqrt{d}/2)W^++\exp(Mhn\sqrt{d}/2)W^-,&n\ \mathrm{even}\\
\exp(-Mh(n+1)\sqrt{d}/2)W^++\exp(Mh(n+1)\sqrt{d}/2)W^-,&n\ \mathrm{odd}
\end{cases}
$$
Abusing notation, write $\mathbf{f}_{j+\lfloor n/2\rfloor}=\mathbf{f}_{\psi(j_1+\lfloor n/2\rfloor,j_2+\lfloor n/2\rfloor,\dots,j_d+\lfloor n/2\rfloor)}$. Define the function $\mathbf{Sf}:\hat{X}\to\mathbb{R}$ by
$$\mathbf{Sf}(x)=\sum_{j=1}^J\mathbf{f}_{j+\lfloor n/2\rfloor}b_j(x).$$
Then $\mathbf{Sf}(x)>0$ for all $x$ and 
$$F(h,n)\mathbf{Sf}(x)\leq \mathbf{Qf}(x)\leq E(h,n)\mathbf{Sf}(x)$$
for all $x\in\hat{X}$. Moreover, if 
\begin{equation}h<\begin{cases}\dfrac{\log|W^+/W^-|}{Mn\sqrt{d}},&n\ \mathrm{even}\\ &\\
\dfrac{\log|W^+/W^-|}{M(n+1)\sqrt{d}},&n\ \mathrm{odd}
\end{cases},
\label{eq:h_cond_pos_d}
\end{equation}
then $\mathbf{Qf}(x)>0$ for all $x\in\hat{X}$.
\label{thm:d_dimhiddenpos}
\end{theorem}
\begin{proof}
For each $j\in \mathcal{J}$, we have
\begin{align*}
Q_j\mathbf{f}&=\sum_{v:W_v>0}W_v\mathbf{f}_{j+v}+\sum_{v:W_v<0}W_v\mathbf{f}_{j+v}
\end{align*}
If $n$ is even, then for each $v$, $\overline{\xi}_{j+v}$ is at a distance at most $hn\sqrt{d}/2$ from $\overline{\xi}_{j+n/2}$. By definition of $K_M$, for any $v$ we have
$$\mathbf{f}_{j+n/2}\exp(-Mhn/2)\leq \mathbf{f}_{j+v}\leq \mathbf{f}_{j+n/2}\exp(Mhn/2)$$
Thus, we can bound $Q_j\mathbf{f}$ above and below by
$$\exp(-Mhn\sqrt{d}/2)\sum_{v:W_v>0}W_v\mathbf{f}_{j+n/2}+\exp(Mhn\sqrt{d}/2)\sum_{v:W_v<0}W_v\mathbf{f}_{j+n/2}$$
$$\leq Q_j\mathbf{f}\leq$$
$$\exp(Mhn\sqrt{d}/2)\sum_{v:W_v>0}W_v\mathbf{f}_{j+n/2}+\exp(-Mhn\sqrt{d}/2)\sum_{v:W_v<0}W_v\mathbf{f}_{j+n/2}$$
Thus for each $j$
$$F(h,n)\mathbf{f}_{j+n/2}\leq Q_j\mathbf{f}\leq E(h,n)\mathbf{f}_{j+n/2}$$
We thus have
$$\mathbf{Qf}(x)=\sum_{j\in\mathcal{J}}(Q_j\mathbf{f})b_j(x)\leq \sum_{j\in \mathcal{J}}E(h,n)\mathbf{f}_{j+n/2}b_j(x)=E(h,n)\mathbf{Sf}(x)$$
A similar calculation shows $F(h,n)\mathbf{Sf}(x)\leq \mathbf{Qf}(x)$. Notice $\mathbf{Sf}(x)>0$ since $\mathbf{f}_{j+n/2}>0$ for all $j\in\mathcal{J}$ and $b_j(x)>0$ for all $j\in\mathcal{J}$ and $x\in\hat{X}$. 

We have $\mathbf{Qf}(x)>0$ so long as $F(h,n)>0$, which, after a straightforward calculation, occurs when
$$h<\frac{\log|W^+/W^-|}{Mn\sqrt{d}}.$$
\end{proof}

If $n$ is odd, then for each $v$, $\overline{\xi}_{j+v}$ is at a distance at most $h(n+1)\sqrt{d}/2$ from $\overline{\xi}_{j+\lfloor n/2\rfloor}$. Repeating the arguments above with $\mathbf{f}_{j+n/2}$ replaced by $\mathbf{f}_{j+\lfloor n/2\rfloor}$ and $\exp(\pm Mhn\sqrt{d}/2)$ replaced by $\exp(\pm Mh(n+1)\sqrt{d}/2)$ proves the result for the odd $n$ case. Again, a calculation shows that $F(h,n)>0$ when 
$$h<\frac{\log|W^+/W^-|}{M(n+1)\sqrt{d}}.$$

\begin{theorem}
Let $\mathbf{f}\in K_M$ and $h>0$ be small enough such that (\ref{eq:h_cond_pos_d}) is satisfied. Then for all $x,y$ in the parameter hyperrectangle of $\mathbf{Qf}$ we have
$$|\log \mathbf{Qf}(x)-\log \mathbf{Qf}(y)|\leq C|x-y|$$
where, if $n$ is even, 
$$C=\frac{(E(h,n)-F(h,n)\exp(-Mh))\sqrt{d}\exp(Mhn\sqrt{d})}{hF(h,n)}$$
and $C\to M\sqrt{d}(1+n\sqrt{d}\|Q\|)$ as $h\to0^+$. If $n$ is odd, then 
$$C=\frac{(E(h,n)-F(h,n)\exp(-Mh))\sqrt{d}\exp(Mh(n+1)\sqrt{d})}{hF(h,n)}$$
and $C\to M\sqrt{d}(1+(n+1)\sqrt{d}\|Q\|)$ as $h\to0^+$.
\label{thm:d_dim_lip}
\end{theorem}
\begin{proof}
We'll begin by deriving bounds for the partial derivatives $|\partial_k(\mathbf{Qf}(x))/\mathbf{Qf}(x)|$ for $x$ in the parameter hyperrectangle. Let $e_k$ denote the unit vector in the direction of the $k$-th axis. By the partial differentiation formula for B-splines, we have
\begin{align*}
\partial_k(\mathbf{Qf})(x)&=\sum_{j\sim x}\frac{1}{h}[Q_j\mathbf{f}-Q_{j-e_k}]b_j^{n-e_k}(x)\\
&\leq\sum_{j\sim x}\frac{1}{h}[E(h,n)\mathbf{f}_{j+\lfloor n/2\rfloor}-F(h,n)\mathbf{f}_{j-e_k+\lfloor n/2\rfloor}]b_j^{n-e_k}(x)\\
&\leq\sum_{j\sim x}\frac{1}{h}[E(h,n)-F(h,n)\exp(-Mh)]\mathbf{f}_{j+\lfloor n/2\rfloor}b_j^{n-e_k}(x)
\end{align*}
Bounding $\mathbf{Qf}(x)$ below by $F(h,n)\mathbf{Sf}(x)$, we have
$$\left|\frac{\partial_k(\mathbf{Qf}(x)}{\mathbf{Qf}(x)}\right|\leq\frac{E(h,n)-F(h,n)\exp(-Mh)}{hF(h,n)}\cdot \frac{\sum_{j\sim x}\mathbf{f}_{j+\lfloor n/2\rfloor} b_j^{n-e_k}(x)}{\sum_{j\sim x}\mathbf{f}_{j+\lfloor n/2\rfloor}b_j^n(x)}$$
Let $j^*$ denote the ``middlemost" index of the $j$ relevant to $x$. When $n$ is even, this is the index corresponding to the midpoint of the mesh element containing $x$, and for all $j$ relevant to $x$, $\mathbf{f}_{j+\lfloor n/2\rfloor}$ is a distance at most $hn\sqrt{d}/2$ away from $\mathbf{f}_{j^*+\lfloor n/2\rfloor}$. Thus
$$\frac{\sum_{j\sim x}\mathbf{f}_{j+\lfloor n/2\rfloor} b_j^{n-e_k}(x)}{\sum_{j\sim x}\mathbf{f}_{j+\lfloor n/2\rfloor}b_j^n(x)}\leq \frac{\exp(Mhn\sqrt{d}/2)\sum_{j\sim x}\mathbf{f}_{j^*+\lfloor n/2\rfloor} b_j^{n-e_k}(x)}{\exp(-Mhn\sqrt{d}/2)\sum_{j\sim x}\mathbf{f}_{j^*+\lfloor n/2\rfloor}b_j^n(x)}=\exp(Mhn\sqrt{d})\frac{\sum_{j\sim x}b_j^{n-e_k}(x)}{\sum_{j\sim x}b_j^n(x)}$$
Since $x$ is in the parameter interval, $\frac{\sum_{j\sim x}b_j^{n-e_k}(x)}{\sum_{j\sim x}b_j^n(x)}=1$, so we obtain
\begin{equation}
\left|\frac{\partial_k(\mathbf{Qf})(x)}{\mathbf{Qf}(x)}\right|\leq\frac{[E(h,n)-F(h,n)\exp(-Mh)]\exp(Mhn\sqrt{d})}{hF(h,n)}
\label{eq:partial_bound}
\end{equation}

Now, for any pair of points $x$ and $y$ in the parameter hyperrectangle, let $[x,y]$ denote the line segment connecting $x$ and $y$. We then have
$$|\log \mathbf{Qf}(x)-\log \mathbf{Qf}(y)|=\left|\int_{[x,y]}\frac{\nabla\mathbf{Qf}(t)}{\mathbf{Qf}(t)}\,dt\right|\leq |x-y|\cdot\max_{t\in[x,y]}\left|\frac{\nabla\mathbf{Qf}(t)}{\mathbf{Qf}(t)}\right|.$$
Using \ref{eq:partial_bound}, we can bound 
$$\left|\frac{\nabla\mathbf{Qf}(t)}{\mathbf{Qf}(t)}\right|\leq \sqrt{d}\cdot\max_{1\leq k\leq d}\left|\frac{\partial_k(\mathbf{Qf}(t))}{\mathbf{Qf}(t)}\right|\leq\frac{[E(h,n)-F(h,n)\exp(-Mh)]\sqrt{d}\exp(Mhn\sqrt{d})}{hF(h,n)}:=C$$
Thus $|\log\mathbf{Qf}(x)-\log\mathbf{Qf}(y)|\leq C|x-y|$ for $x,y$ in the parameter hyperrectangle.

The limit calculations proceed in a nearly identical manner to the one-dimensional case, so we will omit them for brevity.

In the case when $n$ is odd, the argument above also proceeds nearly identically to the even case; the only difference is that $\mathbf{f}_{j+\lfloor n/2\rfloor}$ is a distance at most $h(n+1)\sqrt{d}/2$ away from $\mathbf{f}_{j^*+\lfloor n/2\rfloor}$ for each $j$ relevant to $x$. %We leave the details of the calculations 
\end{proof}

The next theorem is the higher-dimensional analogue to Theorem \ref{thm:1d_cone_mapping}. Since the only portion of the proof of Theorem \ref{thm:1d_cone_mapping} that utilized the fact that the setting was one-dimensional was obtaining the Lipschitz condition for $\mathbf{Qf}$, that same argument can be used to prove the following theorem, so we omit the proof for brevity.

\begin{theorem}
Let $L(v)$ denote the maximal Lipschitz constant among $\{\phi_e\}_{e\in E^v}$. Let $C$ be defined as in the statement of Theorem \ref{thm:d_dim_lip} and take $h>0$ small enough such that (\ref{eq:h_cond_pos_d}) is satisfied. If $v$ is taken large enough so that 
\begin{equation}
C\cdot L(v)\leq M'-M_0(v)s,
\label{eq:cone_map_cond_d}
\end{equation}
then $\mathbf{L}_{s,v}(K_M\setminus\{0\})\subset K_{M'}$.
\label{thm:cone_mapping_d}
\end{theorem}

Finally, we prove the higher-dimensional analogue of Theorem \ref{thm:1d_lemma_setup}.

\begin{theorem}
Let $f_s$ be the strictly positive eigenfunction associated to the Perron-Frobenius operator defined by $(\ref{eq:perron})$ with eigenvalue $\lambda_s$, and let $\mathbf{f_s}\in\mathbb{R}^{N^d}$ be given by $(\mathbf{f_s})_j=f_s(\overline{\xi}_j)$ for $j=\mathcal{J}$. Let $A=\max_{x\in \hat{X}}|\nabla f_s(x)/f_s(x)|$, let $A_{n,k}=\max_{x\in \hat{X}}|\partial_k^{(n+1)}f_s(x)/f_s(x)|$, take $h>0$ small enough such that (\ref{eq:h_cond_pos_d}) is satisfied, and define
$$B=e^{A(2n+1)h\sqrt{d}}c(n,d)\|Q\|(2n+1)^{n+1}\sum_{k=1}^dA_{n,k}$$
where $c(n,d)$ is the constant from Theorem \ref{thm:quasi_error_higher_dim}.
Then for all $x\in \hat{X}$,
\begin{equation}
(1-Bh^{n+1})\lambda_s^vf_s(x)\leq \mathbf{L}_{s,v}\mathbf{f_s}(x)\leq (1+Bh^{n+1})\lambda_s^vf_s(x)
\label{eq:d_dim_setup_1}
\end{equation}
Let $M=A+2/(1-h^2)$ then with $K=K_M$,
\begin{equation}
(1-Bh^{n})\lambda_s^v\mathbf{f_s}\leq_K \mathbf{L}_{s,v}\mathbf{f_s}\leq_K (1+Bh^{n})\lambda_s^v\mathbf{f_s}.
\label{eq:d_dim_setup_2}
\end{equation}
%For $M'$ such that $0<M'<M$, take $v\geq1$ large enough so that (\ref{eq:cone_map_cond}) is satisfied, then Lemma \ref{lem:cones} is 
\label{thm:d_dim_lemma_setup}
\end{theorem}
\begin{proof}
We begin by noting that since $|\nabla f_s/f_s|\leq A$, it follows for all $x,y\in \hat{X}$ that
$$|\log f_s(x)-\log f_s(y)|\leq \int_{[x,y]}\max_{t\in[x,y]}\left|\frac{\nabla f_s(t)}{f_s(t)}\right|\leq A|x-y|.$$
Thus $f_s\in K_{A}$ and in particular, $f_s(x)\leq f_s(y)e^{A|x-y|}$ for all $x,y\in \hat{X}$.

We will prove the rightmost inequalities of (\ref{eq:d_dim_setup_1}) and (\ref{eq:d_dim_setup_2}), since the leftmost inequalities will follow by analogous arguments. To prove (\ref{eq:d_dim_setup_1}), we have
$$\mathbf{L}_{s,v}\mathbf{f_s}(x)=\sum_{e\in E^v}\|D\phi_e(x)\|^s\mathbf{Qf_s}(\phi_e(x))=\sum_{e\in E^v}\|D\phi_e(x)\|^sQf_s(\phi_e(x)),$$
since $\mathbf{Qf}$ and $Qf$ are the same function. By Theorem \ref{thm:quasi_error_higher_dim}, it follows that 
$$Qf_s(\phi_e(x))\leq |Qf_s(\phi_e(x))-f_s(\phi_e(x))|+|f_s(\phi_e(x))|\leq f_s(\phi_e(x))+c(n,d)\|Q\|(2n+1)^{n+1}h^{n+1}\sum_{k=1}^d\|\partial_k^{(n+1)}f_s\|_{\infty,D_{\phi_e(x)}},$$
where $\|\partial_k^{(n+1)}f_s\|_{\infty,D_{\phi_e(x)}}$ is the maximum of $|\partial_k^{(n+1)}f_s|$ over $D_{\phi_e(x)}$, the union of the supports of the splines containing $\phi_e(x)$, which is an hypercube with side length $(2n+1)h$. Let the maximum be achieved at $x^*\in D_{\phi_e(x)}$, then since $|\partial_k^{(n+1)}f_s(x^*)/f_s(x^*)|\leq A_{n,k}$ and $f_s\in K_{A}$, it follows that 
$$|\partial_k^{(n+1)}f_s(x^*)|\leq A_{n,k} f_s(x^*)\leq A_{n,k}e^{A(2n+1)h\sqrt{d}}f(\phi_e(x)).$$
We therefore have
\begin{align*}
\mathbf{L}_{s,v}\mathbf{f_s}(x)&\leq\sum_{e\in E^v}\|D\phi_e(x)\|^s\left[f_s(\phi_e(x))+c(n,d)\|Q\|(2n+1)^{n+1}h^{n+1}f_s(\phi_e(x))\sum_{k=1}^dA_{n,k}\right]\\
&=(1+Bh^{n+1})\sum_{e\in E^v}\|D\phi_e(x)\|^sf_s(\phi_e(x))\\
&=(1+Bh^{n+1})L_{s,v}f_s(x)=(1+Bh^{n+1})\lambda_s^vf_s(x).
\end{align*}
Similar to Theorem \ref{thm:cone_mapping_d}, (\ref{eq:d_dim_setup_2}) can be proved using the exact same argument as (\ref{eq:1d_setup_2}). Since the argument used to prove (\ref{eq:1d_setup_2}) only needed (\ref{eq:1d_setup_1}), and the one-dimensional setting was not utilized after proving (\ref{eq:1d_setup_1}), the argument thereafter generalizes immediately to the higher-dimensional setting. 
\end{proof}

Similar to the one-dimensional case, from Theorem \ref{thm:cone_mapping_d}, we require $L(v)\cdot C\leq M'-M_0(v)s$. Since $C$ approaches a constant multiple of $M$ as $h\to0$, and we can select $M'$ as close to $M$ as we like, so long as $M'<M$, it again follows that the necessary number of iterations is the minimum number $v$ such that $L(v)\cdot C/M<1$. Under this condition, we can select $0<M'<M$ sufficiently large so that the hypotheses of Theorem \ref{thm:cone_mapping_d} are satisfied. %This, in turn, will require $h$ to be smaller, but, as in the one-dimensional case, the conditions on $h$ produced by these choices will be easily satisfied.

\section{Numerical Results}

In this section we provide numerical results to verify our method's accuracy. In all of the examples, we take $E$ to be a finite alphabet.
%We approximate the Perron-Frobenius operator defined in Equation \ref{eq:perron} by a matrix constructed using B-splines. 

\subsection{1D Continued Fractions}

Throughout the following, we let $n$ denote the degree of the B-splines, $h$ denote the mesh size, and we take $\hat{X}=[-nh,1+nh]$.

As discussed at the end of Section 3.1, for a given alphabet $E\subset\mathbb{N}$, our goal is to choose the number of iterations $v$ to be large enough to ensure that $L(v)\cdot C\leq M'-M_0(v)s$. By choosing $h$ to be sufficiently small, we can then select $0<M'<M$ so that the hypotheses of Theorems 3.2 and 3.3 are satisfied, and thus the estimates provided in Theorem 3.4 hold. This choice of $v$ depends on $\min E$, and in fact, the minimal element of $E$ is the controlling factor in determining the required number of iterations. Fortunately, we can determine the number of required iterations for each example numerically.

We'll illustrate the complete process for the example $E=\{1,2\}$ using degree 4 B-splines, taking $h=0.001$. In this case, we have $W^+=251/144$ and $W^-=-107/144$. Theorem 3.4 requires us to take $M$ at least as big as $A_1+2/(1-h^2)$, so it suffices to take $M=5$. Using these values, we obtain $E(h,n)<1.025$ and $F(h,n)>0.975$, which in turn makes $C<57.107$. The restriction on $h$ from Theorem 3.1 becomes $h<0.043$, which is clearly satisfied.

Denote by $\phi_e^v$ the composition of $\phi_e$ with itself $v$ times. Then for the map $\phi_1(x)=1/(x+1)$ on the interval $\hat{X}=[-0.004,0.004]$, we have
$$L(4)=\mathrm{Lip}(\phi^4_1)=\max_{x\in\hat{X}}|[\phi^4_1]'(x)|=|[\phi^4_1]'(-0.004)|<0.041.$$
Next, we determine $M_0(v)$ using the Remark following the proof of Theorem 3.2. We have
$$M_0(4)=\mathrm{Lip}(\log D\|\phi^4_1\|)=\max_{x\in\hat{X}}\left|\frac{[\phi^4_1]''(x)}{[\phi^4_1]'(x)}\right|=\left|\frac{[\phi^4_1]''(-0.004)}{[\phi^4_1]'(0.004)}\right|<1.21$$
We can bound $s$ above by $1$, then
$$L(4)C\leq M'-M_0(4)s\implies (0.041)(57.107)\leq M'-1.21\implies 3.56\leq M'$$
Thus, we can take $M'$ such that $3.56\leq M'<5$ to ensure that Theorem 3.3 is satisfied. Any choice of $h<0.001$ satisfies these requirements, so we can use 4 iterations to compute any example with $\min E=1$ so long as $h<0.001$.

While the requirement on $h$ from Theorem 3.1 is one consideration, there are also numerical requirements that come into play when choosing $h$. In particular, for a given alphabet $E$, $h$ should be chosen to be relatively small in comparison to $1/\max E$ in order to ensure that the grid points sent close to 0 can be distinguished. Thus, $h$ should be chosen to conform to both the theoretical and numerical requirements. 
%Using 4 iterations of $\phi_1(x)=1/(x+1)$, 

Table \ref{table:1D} shows lower and upper estimates of the Hausdorff dimensions of various one-dimensional continued fraction systems using degree 4 B-splines. For all examples, we bounded $B$ using $s=1$.

%Tables \ref{table:conv12}, \ref{table:conv34}, and \ref{table:conv100} show the numerically computed convergence rates of our method for the alphabets $E=\{1,2\}$, $E=\{1,2,\dots,34\}$, and $E=\{1,2,\dots,100\}$ respectively. As we can see, we are achieving the order 3 convergence we expect from the theory. We only present tables for these examples for the sake of brevity, but we have confirmed the convergence rates for all of examples in Table \ref{table:1D}, so long as $h<1/\mathrm{max}E$. 

\begin{table}[h]
\centering
\begin{tabular}{|c|l|}\hline
$\{1,2\}\quad v=4$&$h=10^{-5}$\\\hline
Lower Estimate: & 0.53128 05062 59\\\hline
Upper Estimate: & 0.53128 05062 94\\\hline\hline
$\{1,2\}\quad v=4$&$h=10^{-6}$\\\hline
Lower Estimate: & 0.53128 05062 77203\\\hline
Upper Estimate: & 0.53128 05062 77205\\\hline\hline
$\{1,2,3,4,5\}\quad v=4$&$h=10^{-5}$\\\hline
Lower Estimate: & 0.83682 94436 64\\\hline
Upper Estimate: & 0.83682 94436 98\\\hline\hline
$\{2,4,6,8,10\}\quad v=4$&$h=10^{-5}$\\\hline
Lower Estimate: & 0.51735 70309 27\\\hline
Upper Estimate: & 0.51735 70309 46\\\hline\hline
$\{1,2,\dots,10\}\quad v=4$ & $h=10^{-5}$\\\hline
Lower Estimate: & 0.92573 75911 31\\\hline
Upper Estimate: & 0.92573 75911 61\\\hline\hline
$\{100,10000\}\quad v=1$ & $h=10^{-6}$\\\hline
Lower Estimate: & 0.05224 65926 38658\\\hline
Upper Estimate: & 0.05224 65926 38658\\\hline
\end{tabular}
\caption{Hausdorff dimension estimates for various 1D continued fraction systems using degree 4 B-splines. $v$ denotes the number of iterations.}
\label{table:1D}
\end{table}

\begin{comment}
\begin{table}[h]
\centering
\begin{tabular}{|c|l|}\hline
$\{1,2\}$&$h=10^{-5}$\\\hline
Lower Estimate: & 0.53128 05062 7707\\\hline
Upper Estimate: & 0.53128 05062 7734\\\hline\hline
$\{1,2\}$&$h=10^{-6}$\\\hline
Lower Estimate: & 0.53128 05062 77205\\\hline
Upper Estimate: & 0.53128 05062 77205\\\hline\hline
$\{1,2,\dots,34\}$&$h=10^{-5}$\\\hline
Lower Estimate: & 0.98041 96252 269\\\hline
Upper Estimate: & 0.98041 96252 271\\\hline\hline
$\{1,2,\dots,100\}$&$h=10^{-5}$\\\hline
Lower Estimate: & 0.99366 11108 1055\\\hline
Upper Estimate: & 0.99366 11108 1071\\\hline\hline
$\{100,10 000\}$ & $h=10^{-5}$\\\hline
Lower Estimate: & 0.05224 65926 38646\\\hline
Upper Estimate: & 0.05224 65926 38672\\\hline\hline
$\{\text{Primes $<10,000$}\}$&$h=10^{-5}$\\\hline
Lower Estimate: & 0.67243 17094 03618\\\hline
Upper Estimate: & 0.67243 17094 03699\\\hline
\end{tabular}
\caption{Hausdorff dimension estimates for various 1D continued fraction systems using quadratic B-splines. The estimates for $E=\{1,2\}$ with $h=10^{-6}$ were limited by machine precision.}
\label{table:1D}
\end{table}
\end{comment}
\clearpage

\clearpage
\subsection{2D Continued Fractions}

\begin{table}[h]
\centering
\begin{tabular}{|c|l|}\hline
$\{(1,0),(1,1),(1,-1),(2,0)\}\quad v=4$&$h=0.002$\\\hline
Lower Estimate: & 1.14955 \\\hline
Upper Estimate: & 1.14959 \\\hline\hline
$\{(5,0),(5,1),(5,-1),(6,0)\}\quad v=2$&$h=1/1500$\\\hline
Lower Estimate: & 0.40858 135\\\hline
Upper Estimate: & 0.40858 154\\\hline\hline
$\{(10,0),(10,1),(10,-1),(11,0)\}\quad v=1$&$h=1/1500$\\\hline
Lower Estimate: & 0.29649 894\\\hline
Upper Estimate: & 0.29649 908\\\hline\hline
$\{(100,0),(100,1),(100,-1),(101,0)\}\quad v=1$&$h=1/2000$\\\hline
Lower Estimate: & 0.15042 97623 \\\hline
Upper Estimate: & 0.15042 97662 \\\hline\hline
$\{(1,0),(2,0)\}\quad v=4$ & $h=1/1000$\\\hline
Lower Estimate: & 0.53127 9 \\\hline
Upper Estimate: & 0.53128 1\\\hline
\end{tabular}
\caption{Hausdorff dimension estimates for various 2D continued fraction systems using degree 4 B-splines. $v$ denotes the number of iterations.}
\label{table:2D}
\end{table}

\section*{Acknowledgements}
The author thanks Dmitriy Leykekhman for his guidance and helpful discussions in preparing this manuscript. This work was supported by NSF Grant Number 2513639. Artificial intelligence tools were used in the editing of the manuscript.

\clearpage
\printbibliography

\end{document}